\documentclass[oneside]{amsart}

\usepackage[letterpaper,body={14.0cm,22.0cm}, mag=1000]{geometry}
\usepackage{amssymb}
\usepackage{amsthm}
\usepackage{amscd}
\usepackage{enumitem}
\usepackage{float}
\usepackage{placeins}
\usepackage{caption}

\usepackage{times}
\usepackage{graphicx}
\usepackage{tikz}
\usepackage{tikz-cd}
\usepackage[all,cmtip]{xy}

\numberwithin{equation}{section}
\theoremstyle{plain}

\newtheorem{cor}[equation]{Corollary}
\newtheorem{lemma}[equation]{Lemma}

\newtheorem{proposition}[equation]{Proposition}

\newtheorem{thm}[equation]{Theorem}

\theoremstyle{definition}
\newtheorem{definition}[equation]{Definition}
\newtheorem{example}[equation]{Example}
\newtheorem{remark}[equation]{Remark}

\newcommand{\dlabel}[1]{\ifmmode \text{\ttfamily \upshape [#1] } \else
{\ttfamily \upshape [#1] }\fi \label{#1}}

\newcommand{\C}{\operatorname{C} }

\newcommand{\Z}{\operatorname{Z} }

\newcommand{\Id}{\operatorname{Id}}

\newcommand{\gen}[1]{\left < #1 \right >}
\newcommand{\Aut}{\operatorname{Aut} }
\newcommand{\Autb}{\operatorname{Autb} }

\newcommand{\Inn}{\operatorname{Inn} }

\newcommand{\im}{\operatorname{Im} }
\newcommand{\Ker}{\operatorname{Ker} }

\newcommand{\Map}{\operatorname{Map} }

\newcommand{\Autcent}{\operatorname{Autcent} }
\newcommand{\Hol}{\operatorname{Hol} }

\begin{document}
\setcounter{page}{1}
\title[Symmetric  skew braces and brace systems]
{Symmetric  skew braces and brace systems}

\author{Valeriy G. Bardakov}
\address{Sobolev Institute of Mathematics, pr. ak. Koptyuga 4, Novosibirsk, 630090, Russia  and Novosibirsk State University, Novosibirsk, 630090, Russia
and Novosibirsk State Agrarian University, Dobrolyubova street, 160, Novosibirsk, 630039, Russia and Regional Scientific and Educational Mathematical Center of Tomsk State University, 36 Lenin Ave., Tomsk, Russia.}
\email{bardakov@math.nsc.ru}

\author{Mikhail V. Neshchadim}
\address{Sobolev Institute of Mathematics, pr. ak. Koptyuga 4, Novosibirsk, 630090, Russia and Novosibirsk State University, Novosibirsk, 630090, Russia and
Regional Scientific and Educational Mathematical Center of Tomsk State University, 36 Lenin Ave., Tomsk, Russia.}
\email{neshch@math.nsc.ru}

\author{Manoj K. Yadav}
\address{Harish-Chandra Research Institute, A CI of Homi Bhabha National Institute, Chhatnag Road, Jhunsi, Prayagraj-211 019, India}
\email{myadav@hri.res.in}


\begin{abstract}
For a skew left brace $(G, \cdot, \circ)$, the map $\lambda : (G, \circ) \to \Aut \,(G, \cdot),~~a \mapsto \lambda_a,$
where $\lambda_a(b) = a^{-1} \cdot (a \circ b)$ for all $a, b \in G$,
is a group homomorphism. Then $\lambda$ can also be viewed as a map from $(G, \cdot)$ to $\Aut \, (G, \cdot)$, which, in general, may not be a homomorphism.  A skew left brace will be called $\lambda$-anti-homomorphic ($\lambda$-homomorphic) if   $\lambda : (G, \cdot) \to \Aut \, (G, \cdot)$ is an anti-homomorphism  (a homomorphism).   We mainly study such  skew left braces.  We device a method for constructing a class of binary operations on a given set so that the set with any two such operations constitute a $\lambda$-homomorphic symmetric skew brace. Most of the constructions of symmetric skew braces dealt with in the literature fall in the framework of our construction.   We then carry out various  such constructions on specific infinite groups.

 \end{abstract}
\maketitle

\section{Introduction}

A  triple $(G, \cdot, \circ)$, where $(G, \cdot)$ and $(G, \circ)$ are  groups,    is said to be a \emph{skew left brace} if
 \begin{equation}
 a \circ (b \cdot c) =  (a \circ b) \cdot a^{-1} \cdot  (a \circ c)
 \end{equation}
 for all $a, b, c \in G$, where $ a^{-1}$ denotes the  inverse of $a$ in $(G, \cdot)$. We call  $(G, \cdot)$ the \emph{additive group} and $(G, \circ)$ the \emph{multiplicative  group} of the skew left brace $(G, \cdot, \circ)$. A skew left brace $(G, \cdot, \circ)$ is said to be a \emph{left brace} if $(G, \cdot)$ is an abelian group. By analogy one can define a  skew right brace and right brace. A  triple $(G, \cdot, \circ)$, where $(G, \cdot)$ and $(G, \circ)$ are  groups,    is said to be a \emph{skew right brace} if
 \begin{equation}
(a \cdot b) \circ c =  (a \circ c) \cdot c^{-1} \cdot  (b \circ c)
 \end{equation}
 for all $a, b, c \in G$. A skew right brace $(G, \cdot, \circ)$ is said to be a \emph{right brace} if $(G, \cdot)$ is an abelian group. A  skew  left brace which is also a skew right brace is said to be \emph{two-sided skew brace}.

 In this article we mainly consider  skew left braces. So, we'll mostly suppress the word `left' and only say skew brace(s).

  The concept of  braces was introduced by Rump \cite{R2007} in 2007 in connection with non-degenerate involutive set theoretic solutions of the quantum Yang-Baxter equation.
We remark that, in fact, braces were considered by Kurosh in his lectures \cite[Section 10]{K} of 1969--1970 academic year.  The subject received a tremendous attention of the mathematical community after the work of Rump; see \cite{BCJO18, FC2018,  JKAV21, KSV21, WR2019, AS2018} and the references therein.  Interest in the study of set theoretic solutions of the quantum Yang-Baxter equations was intrigued by the paper \cite{D1992}  of Drinfeld, published in 1992.  Actually, set-theoretic solutions of the Yang-Baxter equation were studied before Drinfeld formulated his question. It was
Joyce \cite{Joyce} and Matveev~\cite{Matveev} who introduced quandles as invariants of knots and links. Every quandle gives a set-theoretic solution of the Yang-Baxter equation.

   The concept of skew  braces was introduced by Guarnieri and Vendramin \cite{GV2017} in 2017 in connection with non-involutive non-degenerate set theoretic solutions of the quantum Yang-Baxter equation. For a skew left brace $(G, \cdot, \circ)$, it was  proved in \cite{GV2017} that the map
$$
\lambda  :  (G, \circ) \to \Aut \, (G, \cdot),~~a \mapsto \lambda_a
$$
is a group homomorphism, where $\Aut \, (G, \cdot)$ denotes the automorphism  group of $(G, \cdot)$ and $\lambda_a$ is given by $\lambda_a(b) = a^{-1} \cdot (a \circ b)$ for all $a, b \in G$.
Obviously, $\lambda$ can also be viewed as a map from $(G, \cdot)$ to $\Aut \, (G, \cdot)$, which, in general, may not be a homomorphism.
We say that  $(G, \cdot, \circ)$ is a $\lambda$-{\it anti-homomorphic skew left brace} ($\lambda$-{\it homomorphic skew left brace)} if the map $\lambda : (G, \cdot) \to \Aut \,(G, \cdot)$ is an anti-homomorphism (a homomorphism). A  $\lambda$-homomorphic skew left brace $(G, \cdot, \circ)$ is said to be a $\lambda$-{\it cyclic skew left brace} if the image $\im \lambda$ is a cyclic subgroup of $\Aut \,(G, \cdot)$.  $\lambda$-homomorphic skew braces of finite order were studied in \cite{CC} and that of infinite order in \cite{BNY22}.

It is a big challenge to describe  all skew  braces of a given order. For example, there are 51 groups of order 32, but there are 1223061  skew  braces and 25281  left braces  of order 32 (see \cite{BNY}). It is not even easy   to describe all skew braces with a given additive group.  So the construction of certain specific classes of skew braces is highly desirable. In the present article our main focus is on the study of the following problem: For a given group $(G, \cdot)$, define  binary operations `$\circ_i$', $i \in I$ such that $(G, \circ_i, \circ_j)$, $i, j \in I \cup \{0\}$ is a  skew  brace, where $\circ_0 = \cdot$.  Various aspects of this problem has already been studied in  \cite{CS21a, CS21b, Chi, Koch20a, Koch22, R19}. 

In Section \ref{BS} we introduce the notion of  brace systems, symmetric brace system, full symmetric brace system and linear brace system.  In Section \ref{antihomo} we observe that $\lambda$-anti-homomorphic skew braces characterize symmetric skew braces. A skew brace $(G, \cdot, \circ)$ is said to be  symmetric if  $(G, \circ, \cdot)$ is also a skew brace. We then find a procedure of constructing a  symmetric brace system.

In Section \ref{lambda} we take up $\lambda$-homomorphic skew braces $(G, \cdot, \circ)$  with $\lambda(G)$ abelian.  We device recursive procedure for constructing symmetric (linear) brace systems on a given  $\lambda$-homomorphic skew brace.  Section \ref{unification} describes a unification method for most of the existing procedures for constructing symmetric skew braces.  Sections \ref{FP} contains many interesting properties of  $\lambda$-homomorphic and $\lambda$-anti-homomorphic skew braces. Specific constructions of skew braces $(G, \cdot, \circ)$ with abelian $\lambda(G)$ are taken up in Section \ref{construction}. In the concluding section we discuss Rota--Baxter groups and skew braces associated with these.

When the additive group of a skew brace is non-abelian, we'll mostly suppress `$\cdot$' and use $ab$ for $a \cdot b$. For  braces,  we denote the additive operation `$\cdot$' by `$+$'. For a skew brace $(G, \cdot, \circ)$, the inverse of $a \in G$ with respect to `$\circ$' is denoted by $\bar a$. For an indexing set $I$, when we deal with many binary operations $\circ_i$, $i \in I$, on $G$ such that $(G, \circ_i)$ is a group, we shall denote the inverse of $a \in G$ with respect to the operation $\circ_i$ by $a^{\circ_i(-1)}$.


 \section{Brace systems} \label{BS}

The concept of a brace block was introduced in \cite{Koch22} and further studied in \cite{CS21b}.  Analogous concept of a multi-brace was introduced in \cite{BG-1}. In this section we intend to introduce a concept, which we will be calling a  brace system,  to generalize the concepts of brace block and multi-brace.

\begin{definition} Let $G$ and $V$ be non-empty sets, $E \subseteq V \times V$.
A brace system on $G$ associated with the pair $(V, E)$ is an algebraic system
$$\mathcal{B}_{V,E}(G) := (G, \{\circ_v\}_{v \in V}),$$ where   $\circ_v : G \times G \to  G$ is a binary operation for all $v \in V$, satisfying the following  conditions:
\begin{enumerate}
\item $(G, \circ_v)$ is a group for all  $v \in V$,
\item $(G, \circ_u, \circ_v)$ is a skew left brace for all pairs $(u, v) \in E$.
\end{enumerate}
\end{definition}

For any non-empty set $V$ and $E \subseteq V \times V$, we can associate  an oriented graph $\Gamma = (V, E')$ with the set of vertices $V$ and there is an oriented edge $e_{uv}$ between $u, v \in V$ if $(u, v) \in E$. More precisely $E' = \{ e_{uv} = (u, v) \in E\}$. Because of this association, the brace system $\mathcal{B}_{V,E}(G)$ can also be denoted as $\mathcal{B}_{\Gamma}(G)$.

Notice that it is a two way passage. For any brace system $\mathcal{B}_{V,E}(G)$ one can construct
 an oriented graph $\Gamma = (V, E')$ with the set of vertices $V$ and there is an edge between $u, v \in V$ if $(G, \circ_u, \circ_v)$ is a skew left brace. The preceding definition shows that each pair $(u, v) \in E$ is an oriented edge. To distinguish,  we denote
 $E' = \{ e_{uv} = (u, v) \in E\}.$

\begin{definition}
Let $\mathcal{B}_{\Gamma}(G)$ be a brace system on $G$ associated to the graph $\Gamma = (V, E)$. Then $\mathcal{B}_{\Gamma}(G)$ is said to be a {\it symmetric brace system} if for each pair $(u, v) \in V \times V$, $e_{uv} \in E$ if and only if $e_{vu} \in E$. A symmetric brace system is said to be a {\it full symmetric brace system} on $G$
if for any pair $(u, v) \in V \times V$, there exists an edge $e_{uv} \in E$.
\end{definition}

One can see that if $\mathcal{B}_{\Gamma}(G)$ is a symmetric brace system, then it defines a non-oriented graph since any edge has an opposite  edge.
If $\mathcal{B}_{\Gamma}(G)$ is a full symmetric brace system, then it defines a full (or complete) non-oriented graph.

\begin{definition} (i) If $V = I$ is a linearly ordered set and $E= \{ (i, j) \in I \times I ~|~ i < j\},$ then
the brace system  $\mathcal{B}_{\Gamma}(G)$ is said to be a {\it linear brace system} and is denoted by $\mathcal{B}_{I}(G)$.

(ii) If $\Gamma = (V, E)$ is a rooted tree, meaning,  there exists a vertex $v_0 \in V$ such that
for any $v \in V$ there exists an oriented path from $v_0$ to $v$, then the brace system
 $\mathcal{B}_{\Gamma}(G)$ is said to be a {\it rooted  brace system}.
\end{definition}

Examples of linear and rooted brace system will be given in upcoming sections.




\section{$\lambda$-anti-homomorphic  skew braces} \label{antihomo}

In this section we introduce the notion of  $\lambda$-anti-homomorphic skew braces and see that these objects characterize symmetric skew braces. For any skew  left brace $(G, \cdot, \circ)$,  it was  proved in \cite{GV2017} that the map
$$
\lambda : (G, \circ) \to \Aut\, (G, \cdot),~~a \mapsto \lambda_a,
$$
where $\lambda_a(b) = a^{-1} (a \circ b)$, is a group homomorphism.
Notice that $\lambda$ can also be viewed as a set map from  $(G, \cdot)$ into $\Aut \,  (G, \cdot)$. Not much is known about the properties of this map. 
The map $\lambda$  together with its domain, the additive group, characterizes the skew brace as a linear $q$-cycle set \cite{R19}.
 We here take up the case when it is an anti-homomorphism, and introduce the following definition.

\begin{definition}
A skew brace $(G, \cdot, \circ)$ is said to be a $\lambda$-{\it anti-homomorphic skew brace} if the map $\lambda : (G, \cdot) \to \Aut \,(G, \cdot)$, defined  by  $\lambda_a(b) = a^{-1} (a \circ b)$, $a, b \in (G,\cdot)$, is an anti-homomorphism.
\end{definition}

The following interesting property follows directly from the definition.

\begin{lemma}
Let $(G, \cdot, \circ)$ be a $\lambda$-anti-homomorphic skew brace. Then for $a \in G$,  $\bar{a} = a^{-1}u_a$ for  some $u_a \in \Ker \, \lambda$, where $\bar{a}$ denotes the inverse of $a$ in $(G, \circ)$.
\end{lemma}
\begin{proof}
As we know $\lambda : (G, \circ) \to \Aut\, (G, \cdot)$,  $a \mapsto \lambda_a$, is a group homomorphism. So by the given hypothesis,   for any $a \in G$, we have
$$\lambda_{\bar a} = \lambda^{-1}_a = \lambda_{a^{-1}},$$
which proves the required assertion.
\end{proof}

We remark that the skew braces  $(G, \cdot, \circ)$ in which the map $\lambda$ is a homomorphism from  $(G, \cdot)$ into $\Aut \,  (G, \cdot)$ have been investigated earlier by some authors \cite{CC}, \cite{BNY22}. We'll come back to it little later. Before we proceed, let us see an example.

\begin{example}\label{example1}
Let $G = (G, \cdot)$ be any group. Define an operation  `$\circ$' on the set  $G$  by the formula
$$
a\circ b=a  \lambda_{a}(b) = ba
$$
for all $a, b \in G$. Then it can be readily  verified that $(G, \cdot, \circ)$ is a skew brace. Now the map $\lambda : G \to \Aut \, G$, $a \mapsto \lambda_a$, where
$$\lambda_a(b) = a^{-1}(a \circ b) = a^{-1} b a \;\; \mbox{for all $b \in G$},$$
is an anti-homomorphism. Hence $(G, \cdot, \circ)$ is a $\lambda$-anti-homomorphic skew brace.
\end{example}

Since the skew braces of the preceding example can be  naturally defined for any group $G$, we would like to give those some name.

\begin{definition}
A skew brace $(G, \cdot, \circ)$ is said to be  {\it natural $\lambda$-anti-homomorphic  skew brace} if  $\circ = \cdot^{op}$.
\end{definition}

We'll show below that  $\lambda$-anti-homomorphic  skew braces can be characterised in terms of  natural $\lambda$-anti-homomorphic  skew brace.

Let $(G, \cdot, \circ)$ is a $\lambda$-anti-homomorphic skew brace. Then for all $a, b, c \in G$,
\begin{eqnarray*}
\lambda_{a \circ b}(c) &=& (a \circ b)^{-1} (a \circ b \circ c)\nonumber\\
&=& (a \circ b)^{-1} (a \circ (b \lambda_b (c)))\nonumber\\
&=& a^{-1} a \lambda_a(\lambda_b(c)) \nonumber\\
&=& \lambda_a(\lambda_b(c)) \nonumber\\
&=& \lambda_{b \cdot a}(c)\;\; \mbox{   (by $\lambda$-anti-homomorphic property)}.
\end{eqnarray*}
Hence
\begin{equation}\label{anti-eq1}
\lambda_{a \circ b} = \lambda_{b \cdot a} \;\; \mbox{ for all $a, b \in G$}.
\end{equation}
Now we recall the definition of a symmetric skew brace, which is also known by the name bi-skew brace.

\begin{definition}
A skew brace $(G, \cdot, \circ)$ is said to be a {\it symmetric skew brace} if  $(G, \circ, \cdot)$ is also a skew brace.
\end{definition}

The necessary and sufficient condition for a skew brace to become symmetric skew brace is given in the following result \cite[Proposition 5.2, Remark 5.5]{BNY22}.

\begin{proposition}\label{prop1}
A skew brace  $(G, \cdot, \circ)$ is  a symmetric skew brace  if and only if
\begin{eqnarray} \label{inc1}
\lambda_{a \circ b } = \lambda_{b \cdot a } \;\; \mbox{for all $a, b \in G$}.
\end{eqnarray}
\end{proposition}



Let $(G, \cdot, \circ)$ be a symmetric skew brace. To differentiate the maps $\lambda$, we write
$\lambda^{\circ}$ for the map $\lambda: (G, \circ) \to \Aut \, (G, \cdot)$ and $\lambda^{\cdot}$ for $\lambda: (G, \cdot) \to \Aut \, (G, \circ)$. Rewriting \eqref{inc1} with the new notation, we get
\begin{equation}\label{eq2}
\lambda^{\circ}_{a \circ b } = \lambda^{\circ}_{b \cdot a } \;\; \mbox{for all $a, b \in G$}.
\end{equation}

Analogously we can also show that
\begin{equation}\label{eq3}
\lambda^{\cdot}_{a \cdot b } = \lambda^{\cdot}_{b \circ a } \;\; \mbox{for all $a, b \in G$}.
\end{equation}

We can now prove

\begin{proposition}\label{anti-prop2}
Every  $\lambda$-anti-homomorphic  skew brace  is symmetric. On the other hand if  $(G, \cdot, \circ)$ is a symmetric skew brace, then  both $(G, \cdot, \circ)$ and $(G, \circ, \cdot)$ are $\lambda$-anti-homomorphic skew braces.
\end{proposition}
\begin{proof}
The first assertion follows from Proposition \ref{prop1} using \eqref{anti-eq1}. Now assume that $(G, \cdot, \circ)$ is a symmetric skew brace. Then \eqref{eq2} and \eqref{eq3} hold true. So by \eqref{eq2} we have
$$\lambda^{\circ}_{a \cdot b } = \lambda^{\circ}_{b \circ a } = \lambda^{\circ}_b \lambda^{\circ}_a \;\; \mbox{for all $a, b \in G$}.$$
This shows that $\lambda^{\circ} : (G, \cdot) \to \Aut \, (G, \cdot)$, $a \mapsto \lambda^{\circ}_a$, is an anti-homomorphism. Hence  $(G, \cdot, \circ)$ is a  $\lambda$-anti-homomorphic skew brace.  That $(G, \circ, \cdot)$ is a $\lambda$-anti-homomorphic skew brace can be seen analogously using \eqref{eq3}.
\end{proof}

\begin{remark}
The preceding proposition characterizes symmetric skew braces in terms of $\lambda$-anti-homorphic skew braces.
\end{remark}

We remark that a similar characterization of symmetric skew braces (under the name bi-skew braces) in terms of $\gamma$-functions has been recently obtained in \cite[Theorem 3.1]{Ca}.

Let us see another example, obtained through exact factorization (\cite[Section 5]{BNY22}). Recall that a group $G$ factorizes through two subgroups $A$ and $B$ if
$$G = A B =\{ ab \mid a \in A, b \in B \}.$$
 The factorization  is said to be exact if $A \cap B = 1$. The following construction  was carried out in \cite{SV}.
Let a group $G$ admit an exact factorization $AB$. Define the operation `$\circ$' on  $G$ by the rule: If $a_1b_1, a_2b_2 \in G$, where $a_1, a_2 \in A$, $b_1, b_2 \in B$,
then
\begin{equation}\label{efeqn1}
(a_1b_1)\circ (a_2b_2)=a_1a_2b_2b_1.
\end{equation}
It is proved in \cite[Theorem 2.3]{SV} that  $(G, \cdot, \circ)$  is a skew brace.  It is easy to see that $(G, \circ)$ is isomorphic to the direct product  $A \times B_{op}$, where  $B_{op}$ is the opposite group of $B$.  If $B$ is abelian, then obviously $(G, \circ) \cong A \times B$. Notice that  `$\circ$' and   `$\cdot$' coincide on $A$.

The  homomorphism  $\lambda : (G, \circ) \to \Aut \, (G, \cdot)$ gives automorphisms $\lambda_z$ of $(G, \cdot)$ for all  $z \in (G, \circ)$ defined by
$$
\lambda_z (y)=z^{-1}(z\circ y), \quad y \in G.
$$
If $z=a_1b_1$, $y=a_2b_2$, where  $a_1,a_2 \in A$, $b_1,b_2 \in B$,
then
$$
\lambda_z (y)=(a_1b_1)^{-1}a_1a_2b_2b_1=b_1^{-1} yb_1.
$$
So, the automorphism  $\lambda_z$ is  the inner automorphism of $G$ induced by the element  $b_1$, where $z = a_1b_1$.

Let  $G = AB$ be as above with $A$ normal in $G$. We  now consider the map $\lambda: G \to \Aut G$ given by $z \mapsto \lambda_z$, where $\lambda_z$ is as defined in the preceding para. A straightforward computation shows that $\lambda$ is an anti-homomorphism. Thus, by Proposition  \ref{anti-prop2}, we have

\begin{proposition}\label{efprop1}
The skew brace $(G, \cdot, \circ)$, constructed above, is   $\lambda$-anti-homomorphic; hence  symmetric.
\end{proposition}

We now device a method for constructing a $\lambda$-anti-homomorphic skew brace from a given anti-homomorphism of a group  into its automorphism group.  Let $G = (G, \cdot)$ be a group and $\lambda : G \to \Aut \, G$ be an anti-homomorphism such that $a \lambda_a(b)a^{-1}b^{-1} \in \Ker \, \lambda$ for all  $a, b \in G$. Then we can define a group $(G, \circ)$ such that, for  $a, b \in G$,
$$a \circ b = a \lambda_a(b).$$
Notice that $\bar a$, the inverse of $a$ with respect to `$\circ$' is $\lambda^{-1}_a(a^{-1})$, where, as usual, $\lambda^{-1}_a$ is the inverse of the automorphism $\lambda_a$. An easy verification now shows that $(G, \cdot, \circ)$ is a $\lambda$-anti-homomorphic  skew brace, and therefore a symmetric skew brace by Proposition \ref{anti-prop2}.  We have proved

\begin{thm}\label{t2}
Let $G = (G, \cdot)$ be a group and $\lambda : G \to \Aut \, G$ be an anti-homomorphism such that $a \lambda_a(b)a^{-1}b^{-1} \in \Ker \, \lambda$ for all  $a, b \in G$. Then $(G, \cdot, \circ)$ is a  $\lambda$-anti-homomorphic skew brace, where $a \circ b = a \lambda_a(b)$ for all $a, b \in G$.
\end{thm}

It now follows that $\lambda: (G, \circ) \to \Aut \,(G, \cdot)$, $a \mapsto \lambda_a$, is a homomorphism, where
$\lambda_a(b) = a^{-1}\cdot (a \circ b)$ for all $a, b \in (G, \circ)$. We first show that $\lambda(G)$ is a subgroup of $\Aut \,(G, \circ)$.  To establish this,  it is  sufficient to show that $\lambda_a(b \circ c) = \lambda_a (b) \circ \lambda_a(c)$ for all $a, b, c \in  (G, \circ)$, which can be easily proved by a direct computation using the fact that $\lambda$ is an anti-homomorphism and $\lambda_{a^{-1}} = (\lambda_a)^{-1}$ for all $a \in G$.

It now follows that  $\lambda : (G, \circ) \to \Aut \,(G, \circ)$, $a \mapsto \lambda_a$, is also a homomorphism. Notice that, as a map, this homomorphism coincides with the map $\lambda : (G, \cdot) \to \Aut \,(G, \cdot)$.
Define a map ${\bar \lambda} : (G, \circ) \to \Aut \,(G, \circ)$, $a \mapsto {\bar \lambda}_a$ by
$${\bar \lambda}_a :=  \lambda_{\bar a}  = \lambda^{-1}_a \;\;\mbox{ for all $a \in (G, \circ)$}.$$
Then it easily follows that $\bar \lambda$ is an anti-homomorphism and
\begin{eqnarray*}
{\bar \lambda}(a \circ {\bar \lambda}_a(b) \circ {\bar a} \circ{\bar b}) &=& \lambda_{\overline{a \circ {\bar \lambda}_a(b) \circ {\bar a} \circ {\bar b}}}\\
&=& \lambda_{b \circ a \circ  \overline{{\bar \lambda}_a(b)} \circ  {\bar a}}\\
&=& \lambda_{a^{-1}  \lambda_{a^{-1}}(b^{-1})  a b}\\
&=& \Id,
\end{eqnarray*}
since $a \lambda_a(b)a^{-1}b^{-1} \in \Ker \, \lambda$ for all  $a, b \in G$. So we can now apply Theorem \ref{t2} to get a symmetric skew brace $(G, \circ, \circ_1)$, where $\circ_1$ is defined as
$$a \circ_1 b := a \circ {\bar \lambda}_a(b) \;\;\mbox{ for all $a, b \in (G, \circ)$}.$$
Notice that for $a, b \in (G, \circ_1)$, we get
$$a \circ_1 b = a \circ  {\bar \lambda}_a(b) = a \cdot \lambda({\bar \lambda}_a(b)) = a \cdot b.$$
Hence $(G, \circ_1) = (G, \cdot)$. We remark that further iteration will be a repetition, and will give no new skew brace.


\bigskip

In the following result  we find a criterion to obtain a (symmetric) skew brace from two $\lambda$-anti-homomorphic skew braces with a common binary operation. For two subgroups $H$ and $K$ of a group $G$, by $[H, K]$ we denote the  subgroup of $G$ generated by the set $\{[h, k] \mid h \in H, k \in K\}$. Please do not get confused with the notation $[G, \lambda(G)]$ defined in \eqref{inc}.

\begin{lemma}\label{anti-link}
Let $(G, \cdot, \circ)$ and $(G, \cdot, \star)$  be  two $\lambda$-anti-homomorphic skew braces such that $[\lambda^{\circ}(G, \circ), \lambda^{\star}(G, \star)] = \Id$  in $\Aut \, (G, \cdot)$. Then $(G, \circ, \star)$ is a skew brace if and only if (i) $[(G, \cdot), \lambda^{\star}(G, \star)] \subseteq \Ker \, \lambda^{\circ}$. Moreover, $(G, \circ, \star)$ is a symmetric skew brace if  and only if in addition to (i)  we also have (ii) $[(G, \cdot), \lambda^{\circ}(G, \circ)] \subseteq \Ker \, \lambda^{\star}$.
\end{lemma}
\begin{proof}
Both $(G, \circ)$ and $(G, \star)$ are groups. For $a \in G$, let $\bar a$ denote the inverse of $a$ in $(G, \circ)$ and $\hat a$ that of $a$ in $(G, \star)$.  Obviously $(G, \circ, \star)$ is a skew brace if and only if
$$a \star (b \circ c) = (a \star b) \circ {\bar a} \circ (a \star c) \;\; \mbox{for all $a, b, c \in G$}.$$
Compute both sides separately in the skew braces $(G, \cdot, \circ)$ and $(G, \cdot, \star)$ simultaneously.  First
\begin{eqnarray*}
a \star (b \circ c) &=& a \star (b \cdot \lambda^{\circ}_b(c))\\
&=& a \cdot \lambda^{\star}_a(b \cdot \lambda^{\circ}_b(c))\\
&=& a  \cdot \lambda^{\star}_a(b) \cdot  \lambda^{\star}_a(\lambda^{\circ}_b(c)).
\end{eqnarray*}
Next
\begin{eqnarray*}
(a \star b) \circ {\bar a} \circ (a \star c) &=& (a \cdot \lambda^{\star}_a(b)) \circ {\bar a} \circ (a \cdot \lambda^{\circ}_a(c))\\
&=& (a \cdot \lambda^{\star}_a(b)) \circ {\bar a}^{-1} \cdot (\bar a \circ \lambda^{\circ}_a(c))\\
&=& (a \cdot \lambda^{\star}_a(b)) \circ \lambda^{\circ}_{\bar a}(\lambda^{\star}_a(c))\\
&=&  a \cdot \lambda^{\star}_a(b) \cdot \lambda^{\circ}_{a \cdot \lambda^{\star}_a(b)}(\lambda^{\circ}_{\bar a}(\lambda^{\star}_a(c)))\\
&=& a \cdot \lambda^{\star}_a(b) \cdot \lambda^{\circ}_{\lambda^{\star}_a(b)}(\lambda^{\star}_a(c)).
\end{eqnarray*}
Now  $(G, \circ, \star)$ is a skew brace if and only if $\lambda^{\star}_a(\lambda^{\circ}_b(c)) = \lambda^{\circ}_{\lambda^{\star}_a(b)}(\lambda^{\star}_a(c))$ if and only if $\lambda^{\circ}_{b^{-1}}(\lambda^{\star}_{a^{-1}} (\lambda^{\circ}_{\lambda^{\star}_a(b)}(\lambda^{\star}_a(c))))  = c$, for all $a, b, c \in G$. By the given hypothesis this is equivalent to $\lambda^{\circ}_{ \lambda^{\star}_a(b) b^{-1}} = \Id$ for all $a, b \in G$, which is further equivalent to $[(G, \cdot), \lambda^{\star}(G, \star)] \subseteq \Ker \, \lambda^{\circ}$. This proves  the first assertion. The second assertion holds by interchanging $\lambda^{\star}_a$ and $\lambda^{\circ}_a$ for $a \in G$ in the above procedure.
\end{proof}

Let us see an example.
\begin{example}
Let $G = (G, \cdot)$ be any group of nilpotency class $3$.  Let $(G, \cdot, \circ)$ be the $\lambda$-anti-homomorphic skew brace of Example \ref{example1}. Obviously $\Ker \, \lambda^{\circ} = \Z(G)$. Now we construct another such skew brace. Let $\beta_x$ be an inner automorphism induced by a fixed element  $x \in G$, i.e., $\beta_x(a) = x^{-1}ax$ for all $a \in G$. Define $\lambda^{\star} : G \to \Aut \,G$, $a \mapsto \lambda^{\star}_a$, where
$$\lambda^{\star}_a(b) = (a^{-1}\beta_x(a))^{-1}ba^{-1}\beta_x(a) \;\;\mbox{for all $b \in G$}.$$
An easy computation shows that $\lambda^{\star}$ is an anti-homomorphism.
Notice that for all $a, b \in G$, we have
$$a \lambda^{\star}_a(b) a^{-1} b^{-1} = [a, b] [b, a^{-1}\beta_x(a)] \in \gamma_2(G).$$
Also notice that $u^{-1}\beta_x(u) \in \gamma_3(G) \le \Z(G)$ for all $u \in \gamma_2(G)$. It now follows  that $a \lambda^{\star}_a(b) a^{-1} b^{-1} \in \Ker \, \lambda^{\star}$. Hence $(G, \cdot, \star)$ is a $\lambda$-anti-homomorphic skew brace such that $\gamma_2(G) \le \Ker \, \lambda^{\star}$. It is now not difficult to see that the hypotheses of  Lemma \ref{anti-link} are satisfied, and hence $(G, \circ, \star)$ is a symmetric skew brace.
\end{example}

\medskip

We now restrict our attention to a special case of $\lambda$-anti-homomorphic skew braces $(G, \cdot, \circ)$ in which $\lambda(G, \cdot)$ is an abelian subgroup of $\Aut \, (G, \cdot)$. In this case $(G, \cdot, \circ)$ becomes a $\lambda$-homomorphic skew brace, a concept introduced and studied in \cite{BNY22}.
We take up this study in the next section.

\bigskip


\section{Brace systems through $\lambda$-homomorphic skew braces} \label{lambda}

In this section,  we first recall the definition of a  $\lambda$-homomorphic skew brace and its connection with regular subgroups of the holomorph of its additive group. We'll then carry out detailed investigation of such skew braces with $\lambda(G)$ abelian subgroup of $\Aut \, G$.

\begin{definition}
A skew brace $(G, \cdot, \circ)$ is said to be a $\lambda$-{\it homomorphic skew brace} if the map $\lambda : (G, \cdot) \to \Aut \,(G, \cdot)$, defined  by  $\lambda_a(b) = a^{-1} (a \circ b)$, $a, b \in (G,\cdot)$, is a homomorphism.
\end{definition}

Let $G$ be a group. The holomorph of $G$ is the group $\Hol G := \Aut  G \ltimes G$, in which the product is given by
$$
(f,a)(g,b)=(fg,af(b))
$$
for all $a,b \in G$ and $f,g \in \Aut G$.
Any subgroup $H$ of $\Hol  G$ acts on $G$ as follows
$$
(f,a) \cdot b = af(b),\quad a,b \in G,\,\, f \in \Aut G.
$$
A subgroup $H$ of $\Hol  G$ is said to be {\it regular} if the action of $H$ on $G$ is free and transitive.  It is  equivalent to the fact that for each $a \in G$ there exists a unique $(f,x) \in H$ such that $xf(1)=a$.
Let $\pi_2 : \Hol G \to G$ denote the projection map from $\Hol G$ onto the second component of $G$. Notice that $\pi_2|_H$, the restriction of $\pi_2$ to $H$, is a bijection.

The following theorem  provides a connection between skew braces and regular subgroups.

\begin{thm}\cite[Theorem 4.2]{GV2017} \label{gv2017}
Let $(G, \cdot, \circ)$ be a skew  brace. Then $\left\{ \, (\lambda_a, a)  \mid  a \in G  \,  \right\}$
is a regular subgroup of $\Hol G$, where $\lambda_a(b) = a^{-1}(a \circ b)$ for all $b \in G$.

Conversely, if $(G, \cdot)$ is a group and $H$ is a regular subgroup of $\Hol  \, (G, \cdot)$, then $(G, \cdot, \circ)$ is a skew  brace such that $(G, \circ) \cong H$, where $a \circ b=a f(b)$
with $(\pi_2|_H)^{-1}(a)=(f, a)\in H$.
\end{thm}

A sufficient condition, in view of Theorem \ref{gv2017}, on  a  group  $G = (G, \cdot)$ so that one can get a $\lambda$-homomorphic  skew brace $(G, \cdot, \circ)$ is given in the following result (\cite[Theorem 2.3]{BNY22}):

\begin{thm} \label{t1}
Let  $G$ be a group, $\lambda : G \to \mathrm{Aut}\, G$, $a \mapsto \lambda_a$,  be a homomorphism of $G$ into the group of its automorphisms. The set
$$ H_{\lambda} := \left\{ \, (\lambda_a, a) \, | \, a \in G  \, \right\}$$
is a subgroup of  $\mathrm{Hol}\, G= \mathrm{Aut}\, G \ltimes G$
if and only if
\begin{equation} \label{inc}
[G, \lambda(G)] := \{ \, b^{-1} \lambda_a (b) \mid a, b \in G\,  \} \subseteq \mathrm{Ker}\,\lambda.
\end{equation}
Moreover, if  $H_{\lambda}$ is a subgroup, then it is  regular, and therefore by Theorem \ref{gv2017} we get a skew brace $(G, \cdot, \circ)$, where `$\circ$' is defined by $a \circ b = a \lambda_a(b)$.
\end{thm}

We remark that this construction can also be carried out directly on the lines of the construction of $\lambda$-anti-homomorphic skew braces. More precisely, without getting into the regular subgroup business,  we can directly prove

\begin{thm}
Let $G = (G, \cdot)$ be a group and $\lambda : G \to \Aut \, G$ be a homomorphism such that $[G, \lambda(G)]  \subseteq \Ker \, \lambda$. Then $(G, \cdot, \circ)$ is a  $\lambda$-homomorphic skew brace, where $a \circ b = a \lambda_a(b)$ for all $a, b \in G$.
\end{thm}

As a consequence we get

\begin{cor}\label{cor-homo}
If $\lambda(G)$ is an abelian subgroup of $\Aut \, G$, then the skew brace $(G, \cdot, \circ)$ constructed in the preceding theorem is symmetric. More precisely, any $\lambda$-homomorphic skew brace  $(G, \cdot, \circ)$ with $\lambda(G)$ abelian is symmetric.
\end{cor}
\begin{proof}
Since $\lambda(G)$ is an abelian, $(G, \cdot, \circ)$ is a $\lambda$-anti-homomorphic skew brace. Assertion now follows from Proposition \ref{anti-prop2}.
\end{proof}

We take a pause to present  examples of a rooted brace system. 
\begin{example} With the preceding setup we can have the following constructions:
\begin{enumerate}
\item Let $(G, \cdot)$ be a group and $\mathcal{H}$  the set of  regular subgroups of $\Hol \, G$. Then for any regular subgroup $H \in \mathcal{H}$ there exists  the operation $\circ_H$ such that $(G, \cdot, \circ_H)$ is a skew  left brace. Hence we have a rooted brace system $\mathcal{B}_{\Gamma}(G)$ with the graph $\Gamma = (V, E)$, where $V = \{v_0\} \cup \{\circ_H\}_{H \in \mathcal{H}}$  and $E = \{v_0\} \times \{\circ_H \}_{H \in \mathcal{H}}$.
\item By applying the construction as in (1) to the group $(G, \circ_H)$ for all $H \in \mathcal{H}$ and continuing this process,
we can construct another rooted  brace system.
\end{enumerate}
\end{example}

Coming back, notice that the skew braces constructed through Theorem \ref{t1} are all $\lambda$-homomorphic.  A variety of examples of such skew braces were constructed in \cite{BNY22}. Let $(G, \cdot, \circ)$ be a $\lambda$-homomorphic  skew brace.  Then $\lambda : (G, \cdot) \to  \Aut \, (G, \cdot)$,  $a \mapsto \lambda_a$, is also a homomorphism. Hence for all $a, b, c \in G$,
\begin{eqnarray*}
\lambda_{a \circ b}(c) &=&  \lambda_a(\lambda_b(c)) \nonumber\\
&=& \lambda_{a \cdot b}(c)\;\; \mbox{   (by $\lambda$-homomorphic property)}.
\end{eqnarray*}
Hence
\begin{equation}\label{eq1}
\lambda_{a \circ b} = \lambda_{a \cdot b} \;\; \mbox{ for all $a, b \in G$}.
\end{equation}
As a consequence, we immediately get
$$[(G, \cdot), \lambda(G, \cdot)]  \subseteq \mathrm{Ker}\,\lambda.$$

Now onwards we concentrate on the case when $\lambda (G, \cdot)$ is an abelian subgroup of $\Aut \, (G, \cdot)$. We remark that in this special situation the concepts of $\lambda$-homomorphic and $\lambda$-anti-homomorphic skew braces coincide.  The following lemma is the base for a recursive process of constructing skew braces.
\begin{lemma}\label{lem1}
Let  $(G, \cdot, \circ)$ be a $\lambda$-homomorphic  skew brace such that $\lambda (G, \cdot)$ is an abelian subgroup of $\Aut \, (G, \cdot)$.  Then the homomorphism $\lambda$ is also a homomorphism from $(G, \circ)$  to  $\Aut \,(G, \circ)$. Moreover  $\lambda(G, \circ)$ is an abelian subgroup of $\Aut \,(G, \circ)$ and  $[(G, \circ), \lambda(G, \circ)] \subseteq \mathrm{Ker}\,\lambda$.
\end{lemma}
\begin{proof}
Since $\lambda (G, \cdot)$ is an abelian subgroup of $\Aut \, (G, \cdot)$, obviously $(G, \cdot, \circ)$ is also a $\lambda$-anti-homomorphic  skew brace. It now follows from the preceding section (discussion after Theorem \ref{t2}) that $\lambda(G)$ is a subgroup of $\Aut \,(G, \circ)$.
Thus the homomorphism $\lambda : (G, \circ) \to \Aut \, (G, \cdot)$ is also a homomorphism from $(G, \circ)$ to $\Aut \, (G, \circ)$ with the image $\lambda(G, \circ)$.

Obviously $ \lambda(G, \circ)$ is abelian.  Now using \eqref{eq1} and the given hypothesis,  we get
$$\lambda({\bar b} \circ \lambda_a(b)) =  \lambda_{ {\bar b} \circ \lambda_a(b)} = \lambda^{-1}_b  \lambda_{\lambda_a(b)} = \lambda_{b^{-1} \lambda_a(b)} = \Id$$
for all $a, b \in (G, \circ)$. Hence $[(G, \circ), \lambda(G, \circ)] \subseteq \mathrm{Ker}\,\lambda$,
which completes the proof.
\end{proof}

We are now going to construct a symmetric linear  brace system  starting with a $\lambda$-homomorphic skew brace  $(G, \cdot, \circ)$  such that $\lambda (G, \cdot)$ is an abelian subgroup of $\Aut \, (G, \cdot)$, and iteratively applying Lemma \ref{lem1}.
Let us denote $\cdot$ and $\circ$ by $\circ_0$ and $\circ_1$ respectively. Notice that $(G, \circ_0, \circ_1)$ is a symmetric skew brace (Proposition \ref{anti-prop2}). By Lemma \ref{lem1} we know that $\lambda : (G, \circ_1) \to \Aut \, (G, \circ_1)$ is a homomorphism, $\lambda(G, \circ_1)$ is abelian and $[(G, \circ_1), \lambda(G, \circ_1)] \subseteq \mathrm{Ker}\,\lambda$. Then by Theorem \ref{t1} we get a $\lambda$-homomorphic skew brace $(G, \circ_1, \circ_2)$ such that for $a, b \in G$,
$$a \circ_2 b := a \circ_1 \lambda_a(b).$$
Again using Proposition \ref{anti-prop2}, we contend that it is symmetric.
By  this iterative process, we can construct, for any integer $i \ge 1$,  a symmetric $\lambda$-homomorphic skew brace $(G, \circ_i, \circ_{i+1})$ such that for $a, b \in G$,
$$a \circ_{i+1} b := a \circ_i \lambda_a(b).$$

We now prove
\begin{lemma}\label{ker-lem}
In the above construction $\lambda$ is a homomorphism and  $\Ker \, \lambda$, as a subset of $G$, is the same at all levels. More precisely, $\Ker \, \lambda$ in $(G, \circ_i)$ is equal to $\Ker \, \lambda$ in $(G, \circ_0)$ for all $i$. Also $\lambda(G, \circ_i) = \lambda(G, \circ_0)$ for all $i$.
\end{lemma}
\begin{proof}
 It is pertinent, by the very construction, that the following diagram is commutative.
 \begin{center}
\begin{tikzcd}
(G, \circ_0)    \arrow[r, "\lambda"]  & \Aut \,(G, \circ_0) \arrow[d, "\Id"]\\
(G, \circ_1)    \arrow[ru, "\lambda^{\circ_1}"] \arrow[r, "\lambda"]  & \Aut \,(G, \circ_1) \arrow[d, "\Id"]\\
(G, \circ_2)    \arrow[ru, "\lambda^{\circ_2}"] \arrow[r, "\lambda"]  & \Aut \,(G, \circ_2) \arrow[d, "\Id"]\\
(G, \circ_3)    \arrow[ru, "\lambda^{\circ_3}"] \arrow[r, "\lambda"]  & \Aut \,(G, \circ_3)\\
\vdots & \vdots,
\end{tikzcd}
\end{center}
where $\Id$ is the identity map on the level of permutations on $G$ and $\lambda$ is a homomorphism at each level.
We start with the fact that $\lambda$ with respect to $\circ_0$ and $\lambda^{\circ_1}$ are equal at element level. Hence
 $\Ker \, \lambda^{\circ_1} = \Ker \, \lambda$. Now by the commutativity of the uppermost triangle, we have that the homomorphisms $\lambda^{\circ_1}$ and $\lambda : (G, \circ_{1}) \to \Aut \, (G, \circ_{1})$ are equal.  Thus the kernels of $\lambda : (G, \circ_{0}) \to \Aut \, (G, \circ_{0})$ and $\lambda : (G, \circ_{1}) \to \Aut \, (G, \circ_{1})$ are the same as subsets of $G$. It also follows that  $\lambda(G, \circ_1) = \lambda(G, \circ_0)$. A straightforward iterative argument now completes the proof.
\end{proof}

\begin{lemma}
If the exponent of $\lambda (G, \circ_0)$ is $n$, then $\circ_{n} = \circ_0$.
\end{lemma}
\begin{proof}
For $a, b \in G$, we have
\begin{eqnarray*}
a \circ_{n} b &=& a \circ_{n-1} \lambda_a(b)\\
 &=& a \circ_{n-2} \lambda^2_a(b)\\
& & \vdots\\
&=& a \circ_0 \lambda^n_a(b)\\
&=& a \circ_0 b.
\end{eqnarray*}
Since this is true for all $a, b \in G$, the assertion holds.
\end{proof}

\begin{lemma}
Let $0 \le i < j$ be integers. Then $(G, \circ_i, \circ_j)$ is a $\lambda$-homomorphic skew brace with $\lambda(G, \circ_j)$ is abelian. Hence $(G, \circ_i, \circ_j)$ is a symmetric skew brace.
\end{lemma}
\begin{proof}
For $a, b \in (G, \circ_j)$, we have
\begin{eqnarray*}
a \circ_{j} b &=& a \circ_{j-1} \lambda_a(b)\\
 &=& a \circ_{j-2} \lambda^2_a(b)\\
& & \vdots\\
&=& a \circ_i \lambda^{j-i}_a(b).
\end{eqnarray*}
As we have seen above $\lambda : (G, \circ_i) \to \Aut \, (G, \circ_i)$ is a homomorphism which is equal, at the element level, to the homomorphism $\lambda : (G, \circ_0) \to \Aut \, (G, \circ_0)$. Set $j-i = k$. Now consider the map $\lambda^k : (G, \circ_i) \to \Aut \,(G, \circ_i)$, $a \mapsto \lambda^k_a$, where  $\lambda^k_a = (\lambda_a)^k$. Since $\lambda(G, \circ_i)$ is abelian, it follows that $\lambda^k$ is a homomorphism. We claim that  $[(G, \circ_i), \lambda^k(G, \circ_i)] \le \Ker \, \lambda^k$.
We already know that $[(G, \circ_i), \lambda(G, \circ_i)] \le \Ker \, \lambda$. So, for $a, b \in G$, we have
\begin{eqnarray*}
\lambda_{b^{\circ_i(-1)} \circ_i \lambda^k_a(b)} & = & \lambda_{b^{\circ_i(-1)} \circ_i \lambda^{k-1}_a(b) \circ_i \big((\lambda^{k-1}_a(b))^{\circ_i(-1)} \circ_i \lambda_a(\lambda^{k-1}_a(b))\big)}\\
&=& \lambda_{b^{\circ_i(-1)} \circ_i \lambda^{k-1}_a(b)}\\
&& \vdots\\
&=& \lambda_{b^{\circ_i(-1)} \circ_i \lambda^{1}_a(b)} = \lambda_{b^{\circ_i(-1)} \circ_i \lambda_a(b)}\\
&=& \Id.
\end{eqnarray*}
Hence $\lambda^k_{b^{\circ_i(-1)} \circ_i \lambda^k_a(b)} = (\lambda_{b^{\circ_i(-1)} \circ_i \lambda^k_a(b)})^k = \Id$. The proof now follows by first applying  Theorem \ref{t1} and then Proposition \ref{anti-prop2}.
\end{proof}

\begin{remark}
Let  $(G, \cdot, \circ)$ be a $\lambda$-homomorphic  skew brace such that $\lambda (G, \cdot)$ is an abelian subgroup of $\Aut \, (G, \cdot)$.  Then $(G, \circ, \cdot)$ is  a  skew brace. Let $\xi : (G, \cdot) \to \Aut(G, \circ)$, $a \mapsto \xi_a$,  be the corresponding map, where $\xi_a(b) = {\bar a} \circ (a \cdot b)$ for all $b \in G$. Notice that $\xi_a = \lambda^{-1}_a$ for all $a \in G$.  Hence $\xi$ is a homomorphism and therefore, $(G, \circ, \cdot)$ is  a  $\xi$-homomorphic skew brace with $\xi(G, \circ)$ abelian in  $\Aut(G, \circ)$.  In the above notation $(G, \circ, \cdot) = (G, \circ_1, \circ_0)$. Define a binary operation  $\circ_{-1}$  on $G$ as follows:
$$a \circ_{-1} b =  a \circ_{0} \xi_a(b) \;\;\; ( = a \circ_{0} \lambda^{-1}_a(b))$$
for all $a, b \in G$.
Then it follows that $(G, \circ_0, \circ_{-1})$ is a symmetric skew brace.   For any integer $n \ge 1$, iteratively define 
$$a \circ_{-n} b =  a \circ_{-n-1} \xi_a(b)$$
for all $a, b \in G$. As above, we see that $(G, \circ_{-i}, \circ_{-j})$ is a symmetric skew brace for all $1 \le i < j$.

\end{remark}

We have proved

\begin{thm}\label{t3}
Let  $G$ be a group, $\lambda : G \to \Aut\, G$, $a \mapsto \lambda_a$,  be a homomorphism such that  $\lambda(G)$ is abelian and $[G, \lambda(G)]  \subseteq \Ker \,\lambda$.
Then there exists a symmetric linear brace system $\mathcal{B}_I(G) := (G, \{\circ_i\}_{i \in I})$ on $G$, $I = \{0, \ldots, n-1\}$, if the exponent of $\lambda(G)$ is $n$. Otherwise $\mathcal{B}_I(G) :=  (G, \{\circ_i\}_{i \in I})$, $I = \mathbb{Z}$, is a symmetric linear brace system on $G$ of infinite size.
\end{thm}

Since a $\lambda$-homomorphic skew brace $(G, \cdot, \circ)$ with abelian $\lambda(G, \cdot)$ is also a $\lambda$-anti-homomorphic skew brace, we can state Proposition \ref{anti-link} as follows.

\begin{proposition}\label{homo-link}
Let $(G, \cdot, \circ)$ and $(G, \cdot, \star)$  be  two $\lambda$-homomorphic skew braces with abelian $\lambda^{\circ}(G, \circ)$ and $\lambda^{\star}(G, \star)$ such that $[\lambda^{\circ}(G, \circ), \lambda^{\star}(G, \star)] = \Id$  in $\Aut \, (G, \cdot)$. Then $(G, \circ, \star)$ is a skew brace if and only if  (i) $[(G, \cdot), \lambda^{\star}(G, \star)] \subseteq \Ker \, \lambda^{\circ}$. Moreover, $(G, \circ, \star)$ is a symmetric skew brace if and only if in addition to (i)  we also have (ii) $[(G, \cdot), \lambda^{\circ}(G, \circ)] \subseteq \Ker \, \lambda^{\star}$.
\end{proposition}

Let us now see connection between different symmetric brace system constructed through Theorem \ref{t3}. When the group $G$ and the linearly ordered set $I$ are clear from the context, for simplicity, we only write $\mathcal{B}^{\circ}$ for the symmetric linear brace system $\mathcal{B}_I(G) :=  (G, \{\circ_i\}_{i \in I})$.

\begin{thm}\label{homo-union}
Let $\mathcal{B}^{\circ} :=  (G, \{\circ_i\}_{i \in I})$  and $\mathcal{B}^{\star} := (G, \{\star_i\}_{i \in I})$ be symmetric linear brace systems on $G$ built on  homomorphisms $\lambda^{\circ} : (G,\circ_{0}) \to \Aut \, (G,\circ_{0})$ and $\lambda^{\star} : (G,\star_{0}) \to \Aut \, (G,\star_{0})$ respectively through Theorem \ref{t3}, where $G = (G, \circ_0) = (G, \star_0)$. Let the hypotheses of Proposition \ref{homo-link} be satisfied for the skew braces $(G, \circ_0, \circ_1)$ and $(G, \star_0, \star_1)$. Then $\mathcal{B}^{\circ} \cup \mathcal{B}^{\star}$ is a symmetric brace system.
\end{thm}
\begin{proof}
For given $i, j \ge 1$, consider the  groups $(G, \circ_i)$ and $(G, \star_j)$. We know that $(G, \circ_0, \circ_i)$ and $(G, \star_0, \star_j)$ are symmetric skew braces. Let $\bar \lambda_i$ and $\tilde \lambda_j$ denote the homomorphisms  $(G, \circ_i) \to \Aut \,(G, \circ_0)$ and $(G, \star_j) \to \Aut \,(G, \star_0)$ respectively. We know that $\bar \lambda_i (a) = (\lambda^{\circ}_a)^i$  and $\tilde \lambda_j (a) = (\lambda^{\star}_a)^j$ for all $a \in G$. It is now clear, using the given hypotheses,  that $\bar \lambda_i : (G, \circ_0) \to \Aut \,(G, \circ_0)$,  $\tilde \lambda_j : (G, \star_0) \to \Aut \,(G, \star_0)$ are homomorphisms and the hypotheses of  Proposition \ref{homo-link} are satisfied  for the braces $(G, \circ_0, \circ_i)$ and $(G, \star_0, \star_j)$. Hence $(G, \circ_i, \star_j)$ is a symmetric skew brace. Since $i, j \ge 1$ are arbitrary, it follows that  $\mathcal{B}^{\circ} \cup \mathcal{B}^{\star}$ is a symmetric brace system.
\end{proof}

We now present a generalization of Proposition \ref{homo-link}. Let $I_1$ and $I_2$ be two linearly ordered set and $I^*_k = I_k \cup \{0\}$, $k = 1, 2$.  Let $\mathcal{B}_{I^*_1}(G) = (G, \{\circ_i\}_{i \in I^*_1})$ and $\mathcal{B}_{I^*_2}(G) =  (G, \{\star_i\}_{i \in I^*_2})$ be two linear brace systems on $G$. Let  $(G, \cdot, \circ_i)$ and $(G, \cdot, \star_j)$ for some $i \in I_1$, $j \in I_2$ be two skew braces, where $\cdot = \circ_0 = \star_0$. Let $\lambda : (G, \circ_i) \to \Aut \, (G, \cdot)$ and $\mu : (G, \star_j) \to \Aut \, (G, \cdot)$ be the corresponding  homomorphisms. With this setup, we have

\begin{lemma}
The algebraic system $(G, \circ_i, \circ_j)$ is a skew brace if
$$
\mu_a (\lambda_b(c)) = \lambda_{a \mu_a(b)}(\lambda_a^{-1}(a^{-1}))
\left( \lambda_{a \mu_a(b) \lambda_{a \mu_a(b)}(\lambda_a^{-1}(a^{-1}))} (a \mu_a(c))  \right)
$$
for any $a, b, c \in G$.
\end{lemma}

\begin{proof}
Obviously $(G, \circ_i)$ and $(G,  \circ_j)$ are groups. We only weed to check the compatibility condition of  the skew brace:
$$
a \circ_j (b \circ_i c) = (a \circ_j b) \circ_i a^{\circ_i(-1)} \circ_i (a \circ_j c),~~~a, b, c \in G,
$$
where $a^{\circ_i(-1)}$ is the inverse to $a$ in the group $(G, \circ_i)$. Computing the left side, we get
$$
a \circ_j (b \circ_i c) = a \mu_a (b \lambda_b(c)) = a \mu_a(b) \mu_a (\lambda_b(c)),
$$
where we used the fact that $\mu_a$ is an automorphism of $(G, \cdot)$.

To compute the right side, notice that the inverse of $a$ in $(G, \circ_i)$ is $a^{\circ_i(-1)} =\lambda_a^{-1}(a^{-1})$, where $a^{-1}$ is the inverse of $a$ in $(G, \cdot)$.
Then the right side is equal to
$$
(a \circ_j b) \circ_i a^{\circ_i(-1)} \circ_i (a \circ_j c) = a \mu_a(b) \lambda_{a \mu_a(b)}(\lambda_a^{-1}(a^{-1}))
\left( \lambda_{a \mu_a(b) \lambda_{a \mu_a(b)}(\lambda_a^{-1}(a^{-1}))} (a \mu_a(c))  \right).
$$
Comparing the left and the right sides, we get the required compatibility condition.
\end{proof}


\section{A unification for known symmetric brace systems}\label{unification}

In this section we present a unification of most of the constructions of symmetric skew braces  available in the literature. The concept was introduced by L.N. Childs in 2019 (\cite{Chi}).  Alan Koch \cite{Koch20a, Koch22} constructed symmetric brace systems on a given group $G$ using abelian endomorphisms of $G$.  This construction was generalized  by  A. Caranti and L. Stefanello \cite{CS21a, CS21b}. We show that all theses symmetric skew braces are $\lambda$-homomorphic.

Let $G = (G, \cdot)$ be a group. Let $\bar A$ be an abelian subgroup of $\bar G := G/\Z(G)$. The inverse image of $\bar A$ in $G$  is $A$. Let $\Map(G, A)$ denote the set of all mappings from $G$ into $A$.  Define
$$\mathcal{A} := \{ f \in \Map(G, A) \mid  f \;\; \mbox{induces an endomorphism from $\bar G$ into $\bar A$}\}.$$
We say that $f_1, f_2 \in \mathcal{A}$ are related if both of these induce the same endomorphism $\bar G$ into $\bar A$. This is an equivalence relation. Let $\bar{\mathcal{A}}$ denote the set of all equivalence classes of the elements of $\mathcal{A}$. The equivalence class of an element $f \in \mathcal{A}$ will be denoted by $[f]$. For a group $G := (G, \cdot)$, we define
$$\mathcal{A}' := \{\alpha : G \times G \to \Z(G) \mid \alpha \mbox{ is bilinear and } \alpha(\Z(G), G) = \alpha(G, \Z(G)) = e\},$$
where $e \in G$ is the identity element. It follows from the definition that  $\alpha(\gamma_2(G), G) = \alpha(G, \gamma_2(G)) = e$.

Let $[f] \in \bar{\mathcal{A}}$ and $\alpha \in \mathcal{A}'$. Define a map $\lambda : G \to \Aut\,G$, $a \mapsto \lambda_a$, where
$$\lambda_a(b) = f(a)^{-\epsilon}bf(a)^{\epsilon}\, \alpha(a, b)$$
with $\epsilon = \pm 1$. We claim that $\lambda$ is a homomorphism and $\lambda(G)$ is an abelian subgroup of $\Aut \,G$. For $a, c \in G$, we know that $f(ac) = f(a)f(c)$ modulo $\Z(G)$. Thus $f(ac) = f(a)f(c)z$ for some  $z \in \Z(G)$. Let $b \in G$. Since $f(G)$ is abelian, we have
\begin{eqnarray*}
\lambda_{a c}(b) &=&  f(a)^{-\epsilon}f(c)^{-\epsilon}z^{-\epsilon}bz^{\epsilon}f(c)^{\epsilon}f(a)^{\epsilon} \, \alpha(ac, b)\\
&=&  f(a)^{-\epsilon}f(c)^{-\epsilon}bf(c)^{\epsilon} \, \alpha(c, b) f(a)^{\epsilon} \, \alpha(a, b)\\
&=& f(a)^{-\epsilon} \lambda_c(b) f(a)^{\epsilon} \, \alpha(a, b)\\
&=& f(a)^{-\epsilon} \lambda_c(b) f(a)^{\epsilon} \, \alpha(a, \lambda_c(b))\\
&=& \lambda_{a}(\lambda_c(b)).
\end{eqnarray*}
This proves that $\lambda_{ac} = \lambda_a \lambda_c$. That $\lambda(G)$ is abelian follows from the fact that $f(G)$ is abelian and $\alpha(G, G) \subseteq \Z(G)$. This settles our claim.

Notice that $\gamma_2(G) \Z(G) \le \Ker \, \lambda$ and $[G, \lambda(G)] \subseteq \gamma_2(G) \Z(G)$. Corollary \ref{cor-homo} now gives

\begin{proposition}
$(G, \cdot, \circ)$ is a symmetric skew brace, where $a \circ b = a \lambda_a(b)$ for all $a, b \in G$.
\end{proposition}

For each pair $([f], \alpha) \in (\bar{\mathcal{A}},  \mathcal{A}')$, by Theorem \ref{t3} we get a symmetric linear brace system, which we denote by $\mathcal{B}_{(f, \alpha)}$. Since any two representatives of the class $[f]$ give rise to the same inner automorphism $\lambda_a$ for each $a \in G$, it follows that $\mathcal{B}_{(f, \alpha)}$ is well defined. We now prove

\begin{thm}
Let $G$ be a group and $A'$ its subgroup such that $A'\Z(G)/\Z(G)$ is abelian. Let $A := A'\Z(G)$ and the corresponding $\bar{\mathcal{A}}$ be as defined above. Then
$$\bigcup_{([f], \alpha) \in (\bar{\mathcal{A}},  \mathcal{A}')} \mathcal{B}_{(f, \alpha)}$$
is a symmetric brace system.
\end{thm}
\begin{proof}
Let $[f_1], [f_2] \in \bar{\mathcal{A}}$ be any two equivalence classes, $\alpha_1, \alpha_2 \in \mathcal{A}'$ and $\mathcal{B}_{(f_i, \alpha_i)}$, $i = 1, 2$, be the corresponding symmetric brace systems. It suffices to show that $\mathcal{B}_{(f_1, \alpha_1)} \cup \mathcal{B}_{(f_2, \alpha_2)}$ is a symmetric brace system. If $\mathcal{B}_{(f_1, \alpha_1)} \subseteq \mathcal{B}_{(f_2, \alpha_2)}$ or $\mathcal{B}_{(f_2, \alpha_2)} \subseteq \mathcal{B}_{(f_1, \alpha_1)}$, then there is nothing to prove. So assume that the remaining case holds. To differentiate, assume that $_{(f_i, \alpha_i)}\lambda : G \to \Aut \,G$ are the homomorphisms induced by $(f_i, \alpha_i)$, $i = 1, 2$.  Let $(G, \cdot, \circ_{(f_1, \alpha_1)})$ be the first skew brace in $\mathcal{B}_{(f_1, \alpha_1)}$ and $(G, \cdot, \circ_{(f_2, \alpha_2)})$ be the first skew brace in $\mathcal{B}_{(f_2, \alpha_2)}$.  As we know that $\gamma_2(G) \le \Ker \, _{(f_i, \alpha_i)}\lambda$ and $[G, \,_{(f_i, \alpha_i)}\lambda(G)] \subseteq \gamma_2(G)$ for $i = 1, 2$.   It is now easy to see that the hypotheses of Theorem \ref{homo-union} are satisfied. Hence $\mathcal{B}_{(f_1, \alpha_1)} \cup \mathcal{B}_{(f_2, \alpha_2)}$ is a symmetric brace system, and the proof is complete.
\end{proof}

\bigskip

\section{Further properties} \label{FP}

In the present section we record various properties on the skew braces considered in  previous sections. Let  $(G, \cdot, \circ)$ be a $\lambda$-homomorphic skew brace such that $b^{-1}\lambda_a(b) \in \Z(G, \cdot)$ for $a, b \in G$. Then it easily follows that $\lambda_a(b)b^{-1} \in \Z(G, \cdot)$.

\begin{proposition}
Let $(G, \cdot, \circ)$ be a $\lambda$-homomorphic skew brace with $\lambda(G)$ abelian. If $[G, \lambda(G)] \subseteq \Z(G)$, then $(G, \cdot, \circ)$ is a two-sided skew brace.
\end{proposition}
\begin{proof}
Let $(G, \cdot, \circ)$ be a skew braces as given in the hypothesis. We need to establish the relation
$$(a \cdot b) \circ c = (a \circ c) \cdot c^{-1} \circ (b \circ c)$$
for all  $a, b,c \in G$. Since $G$ is $\lambda$-homomorphic, we have
\begin{eqnarray*}
(a \cdot b) \circ c &=& (a \cdot b) \cdot \lambda_{a \cdot b}(c)\\
&=&  (a \cdot b) \cdot \lambda_{a}(\lambda_{b}(c))\\
&=&  (a \cdot b) \cdot \lambda_{a}(c\cdot c^{-1} \cdot \lambda_{b}(c))\\
&=&  (a \cdot b) \cdot \lambda_{a}(c) \cdot \lambda_a(c^{-1} \cdot \lambda_{b}(c))\\
&=&  (a \cdot b) \cdot \lambda_{a}(c) \cdot c^{-1} \cdot \lambda_{b}(c)\\
&=&  a \cdot b \cdot (\lambda_{a}(c) \cdot c^{-1}) \cdot b^{-1} \cdot (b \cdot \lambda_{b}(c))\\
&=&  (a  \cdot \lambda_{a}(c)) \cdot c^{-1} \cdot (b \cdot \lambda_{b}(c))\\
&=& (a \circ c) \cdot c^{-1} \cdot (b \circ c).
\end{eqnarray*}
The proof is now complete.
\end{proof}

Let us see some examples. Let $(G, \cdot)$ be a group of nilpotency class $2$ and $(G, \cdot, \circ)$ be any skew brace considered in Section \ref{unification}. Then  $[(G, \cdot), \lambda(G, \cdot)] \subseteq \Ker \, \lambda \le \Z(G)$. Hence all these skew braces are two-sided. {\it One can then construct radical rings from these  two-sided skew braces}.

We now construct symmetric skew braces from $\lambda$-homomorphic skew braces.  For that, we first recall the definition of  opposite skew brace introduced in \cite[Section 3]{R19} and further studied in \cite{KT21}.

\begin{definition}
Let $(G, \cdot, \circ)$ be a skew left brace. Define $\cdot^{op}: G \times G \to G$ by $a \cdot^{op} b = b \cdot a$ for all $a, b \in G$. Then $(G, \cdot^{op})$ is a group and  $(G, \cdot^{op}, \circ)$ is a skew left brace, which is called the {\it opposite skew left brace} of the given skew left brace $(G, \cdot, \circ)$.
\end{definition}

If $\circ = \cdot^{op}$, then $(G, \cdot^{op}, \circ)$ is a trivial skew brace. So we are mainly interested in non-trivial opposite skew braces; $\circ \ne \cdot^{op}$.

\begin{proposition}
Let $(G, \cdot, \circ)$ be a  $\lambda$-homomorphic skew left brace. Then the opposite skew brace $(G, \cdot^{op}, \circ)$ is symmetric if and only if $\Inn \,(G, \cdot)$ centralizes $\lambda \, (G, \cdot)$.  \end{proposition}
\begin{proof}
Since  $(G, \cdot^{op}, \circ)$ is a skew left brace, there exists a homomorphism $\lambda' : (G, \circ) \to \Aut \, (G, \cdot^{op})$, $a \mapsto \lambda'_a$, where $\lambda'_a(b) = a^{-1} \cdot^{op} (a \circ b)$.  We know that $\lambda : (G, \cdot) \to \Aut \, (G, \cdot)$, $a \mapsto \lambda_a$ is a homomorphism, where $\lambda_a(b) = a^{-1} \cdot (a \circ b)$. Notice that

$$\lambda'_a(b) = a^{-1} \cdot^{op} (a \circ b) = (a \circ b) \cdot a^{-1} = \iota_a( \lambda_a(b)),$$
where $\iota_a(x) = a \cdot x \cdot a^{-1}$ for $x \in G$.  Assume that $\Inn \,(G, \cdot)$ centralizes $\lambda \, (G, \cdot)$. Now for $a, b, c \in (G, \cdot^{op})$, we have

\begin{eqnarray*}
\lambda'_{a \cdot^{op} b}(c) &=& \iota_{a \cdot^{op} b} (\lambda_{a \cdot^{op} b}(c))\\
&=& \iota_{b \cdot a} (\lambda_{b \cdot a}(c))\\
&=& \iota_b (\iota_a (\lambda_b (\lambda_a (c))))\\
&=& \iota_b (\lambda_b ( \iota_a(\lambda_a (c))))\\
&=& \lambda'_b (\lambda'_a(c)).
\end{eqnarray*}

This implies that $\lambda' : (G, \cdot^{op}) \to  \Aut \, (G, \cdot^{op})$ is an anti-homomorphism. Hence $(G, \cdot^{op}, \circ)$ is an anti-homomorphic skew brace, and therefore symmetric by Proposition \ref{anti-prop2}. 

The converse part is left as an exercise for the reader.
\end{proof}

It is now not difficult to prove
\begin{cor}
Let $(G, \cdot, \circ)$ be a  $\lambda$-homomorphic skew left brace such that $\lambda (G, \cdot)$ is abelian and $\Inn \,(G, \cdot)$ centralizes $\lambda (G, \cdot)$.  Then $(G, \circ, \cdot^{op})$ is a  two-sided skew brace.
\end{cor}

For a group $G$, let $\Autcent(G) := \C_{\Aut(G)}(\Inn(G))$ be the group of all central automorphisms of $G$. Then any $\lambda$-homomorphic skew left brace $(G, \cdot, \circ)$ with $\lambda(G, \circ) \le \Autcent(G, \cdot)$ gives  rise to a symmetric skew brace, namely $(G, \cdot^{op}, \circ)$. In addition, if $\lambda(G, \circ)$ is abelian, then $(G, \circ, \cdot^{op})$ is a  two-sided skew brace. \emph{ It is tempting to construct  examples of such skew braces in which  $\lambda(G, \circ)$ is not contained in $\Inn(G, \cdot)$. }
\bigskip

Let $(G, \{\circ_i\}_{i \in I})$ be a symmetric linear brace system constructed from the $\lambda$-homomorphic skew brace $(G, \circ_0, \circ_1)$ with $\lambda(G)$ abelian  by the procedure of  Section 4. Then it follows from Lemma \ref{ker-lem} that each group $(G, \circ_i)$ is an extension of $\Ker \lambda$ by the abelian group $\lambda(G)$. It might very well happen that all $(G, \circ_i)$ constitute a single equivalence class of  extensions.  An example of such a symmetric linear brace system can be constructed from the $\lambda$-homomorphic brace $(\mathbb{Z}^2,  +, \circ_1)$ with $(\mathbb{Z}^2, +)$  a free abelian group with free basis $\{x_1, x_2\}$ and $a \circ_1 b = a + \lambda_a(b)$ for all $a, b \in \mathbb{Z}^2$, where $\lambda_{x_i}$ is the automorphism of  $\mathbb{Z}^2$ defined on the basis elements as follows:
\begin{equation*}
\lambda_{x_i} = \left\{
\begin{array}{l}
x_1 \mapsto (1+p) x_1 - p x_2,\\
x_2 \mapsto p x_1 + (1-p) x_2.\\
\end{array}
\right.
\end{equation*}
It follows from \cite[Theorem 4.9 (2)]{BNY22} that $(\mathbb{Z}^2, \circ_1)$ is isomorphic to $(\mathbb{Z}^2,  +)$. Using Lemma \ref{lem1} and then Theorem \ref{t3} we can construct a symmetric linear brace system $\mathcal{B}_I(\mathbb{Z}^2) = (\mathbb{Z}^2, \{\circ_i\}_{i \in I})$  such that $(\mathbb{Z}^2, \circ_i)$ is isomorphic to $(\mathbb{Z}^2,  +)$, where $I = \{0\} \cup \mathbb{N}$ and $\circ_0 = +$. One more such construction is carried out in the next section (Theorem \ref{cyclic1} and its corollary). On the other hand, as we will see again in the next section (Corollaries  \ref{s-last-cor}, \ref{last-cor}), there do exists a symmetric linear  brace system which consists of infinitely many distinct equivalence classes of the extensions mentioned above.

\medskip
We recall  the definition of isomorphism of skew braces.

\begin{definition}
Let $(G, \cdot , \circ)$ and $(H, \cdot , \circ )$
be skew left braces. We say that they are isomorphic if there is a map
$\varphi: G \rightarrow H$ such that
$\varphi: ( G, \cdot ) \rightarrow ( H, \cdot  )$,
$\varphi: ( G, \circ  ) \rightarrow ( H, \circ )$
are group isomorphisms.
\end{definition}

For any  $a,b,c \in G$, the following holds from the definition:
$$
\varphi(a \circ (b \cdot c)) =
(\varphi(a)\circ \varphi(b)) \cdot \varphi(a)^{-1} \cdot (\varphi(a)\circ \varphi(c)).
$$

\begin{definition}
Let $(G, \cdot , \circ) = (H, \cdot , \circ )$, then any brace isomorphism from $(G, \cdot , \circ)$ to $(H, \cdot , \circ )$ is called a {\it brace automorphism} of $(G, \cdot , \circ)$. Let us denote the group of all brace automorphisms of  $(G, \cdot , \circ)$  by  $\Autb \; G$.
\end{definition}

From Lemma \ref{lem1} we have
\begin{proposition}
Let $(G, \cdot , \circ)$ be a $\lambda$-homomorphic skew brace with $\lambda(G)$ is abelian. Then $\lambda(G) \le \Autb \; G$.
\end{proposition}

In a given  symmetric linear brace system $(G, \{\circ_i\}_{i \in I})$, it is not easy to specify whether two  symmetric skew braces $(G, \circ_i , \circ_j)$ and $(G, \circ_{i'} , \circ_{j'})$ are isomorphic or not.
We now establish a criterion for checking such an isomorphism. However, applying the same does not seem easy.

Suppose that we have two $\lambda$-homomorphic
skew left braces
$( G, \cdot , \circ )$ and $( H, \cdot , \circ )$ which
are constructed by homomorphisms  $\lambda: G \rightarrow \mathrm{Aut}\,G$ and
$\mu: H \rightarrow \mathrm{Aut}\,H$, that is,  the operations are defined by the rules
$$
a\circ b=a \lambda_a(b), \quad a, b \in G,
$$
$$
x\circ y=x \mu_x(y), \quad x, y \in H.
$$
Let us find conditions under which these braces are isomorphic.

If $\varphi$ is a brace isomorphism, then
$$
\varphi(ab)=\varphi(a)\varphi(b), \quad
\varphi(a\circ b)=\varphi(a)\circ \varphi(b), \quad a, b \in G.
$$
The second equality has the form:
$$
\varphi(a \lambda_a(b))=\varphi(a)\mu_{\varphi(a)} (\varphi(b)) \Leftrightarrow
\varphi(a) \varphi(\lambda_a(b))=\varphi(a)\mu_{\varphi(a)} (\varphi(b)) \Leftrightarrow
\varphi(\lambda_a(b))=\mu_{\varphi(a)} (\varphi(b)).
$$
We can write the last equality in the form
$$
\varphi \lambda_a \varphi^{-1} \varphi(b)=\mu_{\varphi(a)} (\varphi(b)).
$$
Since $\varphi$ is an isomorphism, we have
$$
\varphi \lambda_a \varphi^{-1} =\mu_{\varphi(a)}, \quad a \in G.
$$
Hence, we have proved

\medskip

\begin{thm}
Two $\lambda$-homomorphic skew  braces
$( G, \cdot , \circ )$ and $( H, \cdot , \circ )$,
which are constructed by homomorphisms $\lambda: G \rightarrow \mathrm{Aut}\,G$ and
$\mu: H \rightarrow \mathrm{Aut}\,H$, respectively, are isomorphic if and only if there is an isomorphism
$$
\varphi: ( G, \cdot ) \rightarrow ( H, \cdot )
$$
such that
$$
\varphi \lambda_a \varphi^{-1} =\mu_{\varphi(a)} \quad \mbox{for all } a \in G.
$$
\end{thm}

\medskip

We  conclude this section by generalizing  the concept of meta-trivial skew brace introduced in \cite{BNY22} by defining  a poly-trivial skew brace.  Let us first recall the definition of an ideal of a skew brace (see \cite{GV2017}).
A non-empty subset $I$ of a skew left brace $(G,\cdot,\circ)$ is said to be an {\it ideal} of $G$ if

1) $\lambda_a(I) \subseteq I$ for all $a \in G$,

2) $I$ is normal in $(G, \cdot)$,

3) $I$ is normal in~$(G, \circ)$.

For example, the kernel of any skew left brace homomorphism is an ideal.
The notion of ideal allows us to consider the quotient skew left brace~$G/I$. Note that the condition~1) is equivalent to the equality $a\circ I = aI$ for all $a\in G$. A skew left brace $(G,\cdot,\circ)$ is said to be {\it meta-trivial} if it admits a ideal $I$ such that both $I$ and the quotient skew skew brace $G/I$ are trivial. This concept can be generalised as follows:

\begin{definition}
A~skew left brace $(G,\cdot,\circ)$ is said to be a {\it poly-trivial skew brace} if there is a  sequence of ideals of $G$:
$$
\{e\} = I_0 \leq I_1 \leq \ldots \leq I_n = G,
$$
such that $I_{i+1} / I_i$ is a trivial skew brace for all $i = 0, 1, \ldots, n-1$. The minimal $n$ with this property is called the {\it step of triviality} of $G$ and is denoted by $st(G)$.
\end{definition}

Notice that a meta-trivial skew brace is  poly-trivial of step of triviality $2$.  All $\lambda$-homomorphic skew braces are meta-trivial (\cite[Theorem 2.12]{BNY22}). We now show that the converse of this statement (\cite[Question 2.13]{BNY22})  is not true.  For any non-trivial group $G:=(G, \cdot)$ with abelian $\gamma_2(G)$ and trivial $\Z(G)$, let  $(G, \cdot, \circ)$, where $\circ = \cdot^{op}$, be  $\lambda$-anti-homomorphic  skew  brace (see Example \ref{example1}).  It follows from \cite[Proposition 2.3]{CSV19} that  $(G, \cdot, \circ)$ is meta-trivial with $G^2 = \gamma_2(G)$, where $G^2 = \gen{\lambda_a(b) b^{-1} \mid a, b \in G}$. Since $\lambda(G) \cong \Inn(G)$ is not abelian,  $(G, \cdot, \circ)$ can not be $\lambda$-homomorphic. The simplest such example is $S_3$, the symmetric group on $3$ letters, which was pointed out by L. Vendramin in the Zentralblatt and AMS review of \cite{BNY22} and also in \cite[Example 3.15]{ST22}.

We now discuss poly-triviality of  $\lambda$-anti-homomorphic  (symmetric) skew braces.  We start by noting that any $\lambda$-anti-homomorphic skew brace with $\Ker \, \lambda = 1$ is natural. Further we have

\begin{proposition}\label{natural}
Let    $(G, \cdot, \circ)$ be a $\lambda$-anti-homomorphic  skew brace. Then either  $(G, \cdot, \circ)$ itself   or  $G/\Ker \, \lambda$ is a natural  $\lambda$-anti-homomorphic  skew brace.
\end{proposition}
\begin{proof}
Suppose that $(G, \cdot, \circ)$ is not a natural  $\lambda$-anti-homomorphic  skew brace.  Notice that $\Ker \, \lambda$ is a normal subgroup of both $(G, \cdot)$ and $(G, \circ)$. Also 
$$\Ker \, \lambda = \{a \in G \mid a \circ b = a \cdot b \mbox{ for all } b \in G\}.$$
For simplicity, let us denote this set by $G_0$. Obviously $G_0$ is a trivial skew brace. Set $\bar{G} = G/G_0$. Then  $(\bar{G}, \cdot)$ and $\bar G,  \circ)$ are groups such that
$$G_0 a \cdot G_0 b = G_0 a \cdot b \,\, \mbox{ and  } \,\, G_0 a \circ G_0 b = G_0 (a \circ b)$$
for all $a, b \in G$.

Since $(G, \cdot, \circ)$ is a $\lambda$-anti-homomorphic skew brace, it follows that 
$$a\lambda_a(b)a^{-1}b^{-1} \in \Ker\, \lambda$$
for all $a, b \in G$. Hence
$$G_0 a \circ G_0 b = G_0 (a \circ b) = G_0 (a\lambda_a(b) a^{-1}b^{-1} ba = G_0 ba = G_0 b \cdot G_0 a$$
for all $a, b \in G$. This proves that $G/\Ker \, \lambda$ is a natural  $\lambda$-anti-homomorphic skew brace.
\end{proof}

As a consequence of the preceding result we conclude 

\begin{cor}
Any $\lambda$-anti-homomorphic  skew brace is either natural or trivial-by-natural. 
\end{cor}

So, to characterise  $\lambda$-anti-homomorphic  skew braces, it only needs to study  natural $\lambda$-anti-homomorphic  skew braces.  Note that the sub-skew brace $G_0 = \Ker \, \lambda$ of a $\lambda$-anti-homomorphic  skew brace $(G, \cdot, \circ)$  as defined in the proof of Proposition \ref{natural} is an ideal of $(G, \cdot, \circ)$. For, since $a\lambda_a(b)a^{-1}b^{-1} \in \Ker\, \lambda$, for any $a \in G$ and $b \in G_0$, $\lambda_a(b) \in G_0$. So if $G/\Ker \, \lambda$ is poly-trivial with $st(G/\Ker \, \lambda) = n$, then $(G, \cdot, \circ)$ is poly-trivial with $st(G)$ at most  $n+1$.

Let us now specialize to natural $\lambda$-anti-homomorphic  skew braces. Suppose $(G, \cdot, \circ)$ is such a skew brace with $\gamma_2(G, \cdot)$ abelian. Then as discussed above $(G, \cdot, \circ)$ is meta-trivial. So we have

\begin{cor}
Let $(G, \cdot, \circ)$ be a  $\lambda$-anti-homomorphic  skew brace such that $\gamma_2(\bar G)$ is abelian, where $\bar G = G/ \Ker \, \lambda$ is the group with binary operation induced by $\cdot$. Then $st(G)$ is at most $3$.
\end{cor}

 Now let $(G, \cdot, \circ)$ be a  natural $\lambda$-anti-homomorphic  skew brace  such that $(G, \cdot)$ is nilpotent with nilpotency class $n$. Then $I_1 := \Ker \, \lambda$ is precisely $\Z(G, \cdot)$, the center of $(G, \cdot)$ and is an ideal of $(G, \cdot, \circ)$. It follows that $(G/I, \cdot, \circ)$ is a natural $\lambda$-anti-homomorphic  skew brace. An easy iterative argument now shows that $(G, \cdot, \circ)$ is poly-trivial with $st(G)$ at most $n$. We conclude this section with the remark that if $(G, \cdot)$ is a non-abelian simple group, then  the natural $\lambda$-anti-homomorphic  skew brace $(G, \cdot, \circ)$ can not be poly-trivial. It is a matter of further investigation to get more information on such skew braces.  After this article was written, there followed an article \cite{ST22} in which  symmetric skew braces, under the name bi-skew braces, are  further invested and some interesting structural results are obtained.  Among other things, the following conjecture of N. Byott \cite{Byott15} (based on the last sentence of Section 1) is  proved for symmetric skew braces: If $(G, \cdot, \circ)$ is a skew brace with $(G, \cdot)$ solvable, then $(G, \circ)$ is solvable.


\bigskip

\section{Some constructions} \label{construction}

In previous sections we considered constructions of  brace systems $(G, \{\circ_i\}_{i \in I})$ for $I = \{0\} \cup \mathbb{N}$. In this section we 
 construct some concrete   symmetric brace systems $(F, \{\circ_i\}_{i \in I})$, where $F$ is a  non-abelian free group and $I$ is a subset of $\mathbb{Z}$. 

We start with the definition of a $\lambda$-cyclic skew brace.

\begin{definition}
A $\lambda$-homomorphic skew brace $(G, \cdot, \circ)$ is said to be {\it $\lambda$-cyclic} if $\lambda(G)$ is cyclic.
\end{definition}

Let $F$ be a non-abelian free group with a basis $X$ and $\theta \in \Aut F$ be an automorphism. We can define a homomorphism $\lambda : F \to \Aut F$ by the rule
$$
\lambda_x = \lambda(x) = \theta ~~\mbox{for all}~x \in X.
$$
Hence, for $a \in F$ its image $\lambda_a  = \theta^{l(a)}$, where $l(y)$, for an element  
$$
y = x_{i_1}^{\alpha_1} x_{i_2}^{\alpha_2} \ldots x_{i_k}^{\alpha_k},~~x_j \in X,~~\alpha_j \in \mathbb{Z}
$$ of $F$ is given by
$l(y) := \alpha_1 + \alpha_2 + \ldots + \alpha_k$. It is not difficult to see that the hypotheses of Theorem \ref{t3} are satisfied.
If $\theta$ has finite order $n$, then we get  symmetric linear brace system $\mathcal{B}_{I}(F) = (F, \{\circ_i\}_{i \in I})$, where $I = \{0, \ldots, n-1\}$. , and  $\circ_0 = \cdot$, $\circ_1 = \circ$. If the automorphism $\theta$ has infinite order, then we can construct  symmetric linear brace system with $I = \{0\} \cup \mathbb{N}$. 

In the rest of this section we analyse both of these possibilities. At first, we consider the case when $F=F_n$ has rank $n$  and $\theta$ is the cyclic permutation of the generators. We prove that in this case $(F_n, \circ)$ is isomorphic to $F_n$, but the operation $\circ$ is different from $\cdot$ and $\cdot^{op}$.

\begin{thm}\label{cyclic1}
Let $F_n = \langle x_1, x_2, \ldots, x_n \rangle$ be  the free group of rank $n$ and $\theta$ be the automorphism which is the permutation of the generators,
$$
\theta: x_1 \mapsto x_2 \mapsto \ldots \mapsto x_n \mapsto x_1.
$$
A homomorphism
$\lambda: F_n \rightarrow \mathrm{Aut}\,F_n$ is defined by the rule
$$
\lambda: x_i \mapsto \theta,\quad  i=1,\ldots,n.
$$
Then in the corresponding   $\lambda$-cyclic skew brace $(F_n, \cdot, \circ)$ the multiplicative group   $( F_n, \circ )$
is an extension of  $\Ker\, \lambda$ that is a free group 
of rank  $n^2-n+1$ with free generators,
$$
z_{j,k}=x_1^k x_j x_1^{-k-1}, \quad k=0,\ldots,n-2, \quad j=2,\ldots,n,
$$
$$
y_{i}=x_1^{n-1} x_i, \quad i=1,\ldots,n
$$
by an infinite cyclic group   $\langle s \rangle$, which acts, by conjugation, on the generators of
$\Ker \, \lambda$ as follows:
\begin{eqnarray*}
s^{-1}y_1s &=& z_{n0}z_{n1}\cdot \ldots \cdot z_{n,n-2}y_{n},\\
s^{-1}y_2s &=& z_{n0}z_{n1}\cdot \ldots \cdot z_{n,n-3}y_{n},\\
s^{-1}y_i s &=& z_{n0}z_{n1}\cdot \ldots \cdot z_{n,n-3}z_{i-2,n-2}y_n, \quad i=3,\ldots,n,\\
s^{-1}z_{j0}s &=& y_n^{-1}y_{j-1}, \quad j=2,\ldots,n,\\
s^{-1}z_{j1}s &=& z_{j-1,0}z_{n,0}^{-1},  \quad j=2,\ldots,n,\\
s^{-1}z_{j,k}s &=&
(z_{n0}z_{n1}\cdot\ldots \cdot z_{n,k-2})(z_{j-1,k-1}z_{n,k-1}^{-1})(z_{n0}z_{n1}\cdot\ldots \cdot z_{n,k-2})^{-1},\\
\end{eqnarray*}
where $j=2,\ldots,n$, $k=2,\ldots,n-2$ and we assume $z_{1,k}=1$.
Moreover,
$$
s^n=z_{21}z_{32}\cdot \ldots \cdot z_{n-1,n-2}z_{n,n-1}.
$$
\end{thm}

\begin{proof}
Take the subgroup
$$
H_{\lambda}=
\left\{ \, (\lambda(a),a) \, | \,  a \in   F_n   \, \right\}
$$
of the holomorph $\mathrm{Hol}\, F_n$.

The kernel of $\lambda$ consists of the words  $w=x_{i_1}^{\alpha_1}\ldots x_{i_m}^{\alpha_m} \in F_n$ such that
$l(w) = 0$, where
$$
l(w) =\alpha_1 + \cdots + \alpha_m.
$$

Notice that for $1 \le i \le n-1$, we have
$$
x_i^{-1}\lambda_{x_j} (x_i)=x_i^{-1}\theta(x_i)=x_i^{-1}x_{i+1}
$$
and for $i =n$, we have
$$x_n^{-1}\lambda_{x_j} (x_n)=x_n^{-1}\theta(x_n)=x_n^{-1}x_{1}$$
for all  $1 \le j \le n$.
Hence, $x_i^{-1}\lambda_{x_j} (x_i)$
lies in the kernel of  $\lambda$. It now follows from \cite[Proposition 2.14]{BNY22} that  $b^{-1}\lambda_a(b) \in Ker \lambda$ for all $a, b \in F_n$.
It is now not difficult to see that the hypotheses of \cite[Themrem 3.5]{BNY22} are satisfied; hence
 $H_{\lambda}$ is a regular subgroup of  $\mathrm{Hol}\, F_n$. This establishes the existence of the $\lambda$-cyclic skew brace $(F_n, \cdot, \circ)$, where $a \circ b = a \cdot \lambda_a(b)$ for all $a, b \in F_n$.

Using the Reidemeister-Shraier method let us find the generators of $\Ker \, \lambda$.

We have

$\bullet$ $\left\{x_1,\ldots,x_n \right\}$ is a generating set of $F_n$,

$\bullet$ $\left\{1,x_1,x_1^2,\ldots, x_1^{n-1}  \right\}$ is a Shraier set of coset representatives
of $F_n/ \mathrm{Ker}\, \lambda$.

Then the non-trivial elements of the form
$x_1^k x_j (\overline{x_1^k x_j})^{-1}$,
$ k=0,\ldots,n-1$, $j=1,\ldots,n$ are free generators of  $\mathrm{Ker}\, \lambda$, where $\bar a$ denotes the inverse of $a$ in $(F_n, \circ)$.
Note that $\mathrm{rank} (\mathrm{Ker}\, \lambda)=1+(n-1)n=n^2-n+1$.
Hence, $\mathrm{Ker}\, \lambda$ has the basis consisting of the elements
$$
z_{j,k} := x_1^k x_j x_1^{-k-1}, \quad k=0,\ldots,n-2, \quad j=2,\ldots,n,
$$
$$
y_{i} := x_1^{n-1} x_i, \quad i=1,\ldots,n.
$$

The quotient  $H_{\lambda}/ \mathrm{Ker}\,\lambda \cong \mathbb{Z}_n $
and is generated, for example, by the image of  $s=(\theta, x_1)$.
Let us find the action of  $s$ on the generators of $\mathrm{Ker}\,\lambda$.

At first, we remark that
$$
s^{-1}=(\theta, x_1)^{-1}=(\theta^{-1}, \theta^{-1}(x_1^{-1}))=(\theta^{-1}, x_n^{-1}).
$$
Hence,
\begin{eqnarray*}
s^{-1}z_{j,k}s &=&(\theta^{-1}, x_n^{-1})(1,x_1^k x_j x_1^{-k-1} ) (\theta, x_1)\\
&=& (\theta^{-1}, x_n^{-1})(\theta,x_1^k x_j x_1^{-k} )\\
&=&  (1, x_n^{-1} x_n^k x_{j-1} x_n^{-k} )\\
&=& (1, x_n^{k-1}x_{j-1} x_n^{-k} ),\\
s^{-1}y_{1}s &=&(\theta^{-1}, x_n^{-1})(1,x_1^{n} ) (\theta, x_1)\\
&=&  (\theta^{-1}, x_n^{-1})(\theta,x_1^{n+1})\\
&=& (1, x_n^{n} )
\end{eqnarray*}
and for  $i>1$,

\begin{eqnarray*}
s^{-1}y_{i}s&=&(\theta^{-1}, x_n^{-1})(1,x_1^{n-1} x_i ) (\theta, x_1)\\
&=& (\theta^{-1}, x_n^{-1})(\theta,x_1^{n-1} x_i x_1)\\
&=&(1, x_n^{n-2} x_{i-1} x_n).
\end{eqnarray*}

So,
$$
s^{-1}z_{j,k}s=x_n^{k-1}x_{j-1} x_n^{-k}, \quad k=0,\ldots,n-2, \quad j=2,\ldots,n,
$$
$$
s^{-1}y_{1}s=x_n^{n},
$$
$$
s^{-1}y_{i}s=x_n^{n-2} x_{i-1} x_n, \quad i=2,\ldots,n.
$$

Further, we need to rewrite the right hand sides in the generators  $z_{j,k}$, $x_i$.

We have
\begin{eqnarray*}
s^{-1}z_{20}s &=& x_n^{-1}x_1=y_n^{-1}y_1,\\
s^{-1}z_{21}s &=& x_1x_n^{-1}=(x_n x_1^{-1})^{-1}=z_{n0}^{-1},\\
s^{-1}z_{22}s &=& x_nx_1x_n^{-2}=z_{n0}z_{n1}^{-1}z_{n0}^{-1},\\
s^{-1}z_{23}s &=& x_n^2x_1x_n^{-3}=z_{n0}z_{n1}z_{n2}^{-1}z_{n1}^{-1}z_{n0}^{-1},\\
 &&\vdots\\
s^{-1}z_{2,n-2}s &=& x_n^{n-3}x_1x_n^{2-n}\\
&= &(z_{n0}z_{n1}\cdot\ldots \cdot z_{n,n-4})z_{n,n-3}^{-1}(z_{n0}z_{n1}\cdot\ldots \cdot z_{n,n-4})^{-1}.
\end{eqnarray*}

Further,
\begin{eqnarray*}
s^{-1}z_{30}s &=& x_n^{-1}x_2 = y_n^{-1}y_2, \\
s^{-1}z_{31}s &=& x_2x_n^{-1}=(x_n x_1^{-1})^{-1}=z_{20}z_{n0}^{-1},\\
s^{-1}z_{32}s &=& x_nx_2x_n^{-2}=z_{n0}z_{21}z_{n1}^{-1}z_{n0}^{-1},\\
s^{-1}z_{33}s &=& x_n^2x_2x_n^{-3}=z_{n0}z_{n1}(z_{22}z_{n2}^{-1})z_{n1}^{-1}z_{n0}^{-1},\\
&& \vdots\\
s^{-1}z_{3,n-2}s &=&
(z_{n0}z_{n1}\cdot\ldots \cdot z_{n,n-4})(z_{2,n-3}z_{n,n-3}^{-1})(z_{n0}z_{n1}\cdot\ldots \cdot z_{n,n-4})^{-1}.
\end{eqnarray*}

Now for any $3 \le j \le n$ and $2 \le k \le n-2$, on the same line, we can compute
\begin{eqnarray*}
s^{-1}z_{j0}s &=& x_n^{-1}x_{j-1}=y_n^{-1}y_{j-1},\\
s^{-1}z_{j1}s &=& x_{j-1}x_{n}^{-1}=z_{j-1,0}z_{n,0}^{-1},\\
&& \vdots \\
s^{-1}z_{j,k}s &=&
(z_{n0}z_{n1}\cdot\ldots \cdot z_{n,k-2})(z_{j-1,k-1}z_{n,k-1}^{-1})(z_{n0}z_{n1}\cdot\ldots \cdot z_{n,k-2})^{-1},
\end{eqnarray*}

Since $z_{1k}=1$,  the action of $s$ on it is trivial. To summarize
\begin{eqnarray*}
s^{-1}z_{j0}s &=& y_n^{-1}y_{j-1},\\
s^{-1}z_{j1}s &=& z_{j-1,0}z_{n,0}^{-1},\\
&& \vdots \\
s^{-1}z_{j,k}s &=&
(z_{n0}z_{n1}\cdot\ldots \cdot z_{n,k-2})(z_{j-1,k-1}z_{n,k-1}^{-1})(z_{n0}z_{n1}\cdot\ldots \cdot z_{n,k-2})^{-1},
\end{eqnarray*}
where $j=2,\ldots,n$, $k=2, \ldots, n-2$.

Further, we consider the action of $s$ on $y_i$ and compute $s^{-1}y_i s$, $i=1, \ldots, n$.
First we establish the  formulas for
$$
x_1^{-k}z_{n0}x_1^{k}, \quad k=1,\ldots,n-1.
$$
We have
\begin{eqnarray*}
x_1^{-1}z_{n0}x_1 &=& x_1^{-1}x_n=x_1^{-n}x_1^{n-1}x_n=y_1^{-1}y_n,\\
x_1^{-2}z_{n0}x_1^{2} &=& x_1^{-1}y_1^{-1}y_n x_1=y_1^{-1}x_1^{-1}y_n x_1=y_1^{-1}z_{n,n-2}y_1,\\
x_1^{-3}z_{n0}x_1^{3} &=& y_1^{-1}z_{n,n-3}y_1,\\
&& \vdots\\
x_1^{1-n}z_{n0}x_1^{n-1} &= & y_1^{-1}z_{n,1}y_1.
\end{eqnarray*}

Hence,
\begin{eqnarray*}
s^{-1}y_1s &=& x_n^n = (z_{n0}x_1)^n=(z_{n0}x_1)(z_{n0}x_1)\cdot \ldots \cdot (z_{n0}x_1)\\
&=& z_{n0}(x_1z_{n0}x_1^{-1})(x_1^{2}z_{n0}x_1^{-2})\cdot \ldots \cdot (x_1^{n-1}z_{n0}x_1^{1-n})x_1^{n}\\
&=& z_{n0}z_{n1}\cdot \ldots \cdot z_{n,n-2}y_{n},\\
s^{-1}y_2s &=& x_n^{n-2}x_1x_n=(z_{n0}x_1)^{n-2}x_1x_n=(z_{n0}x_1)\cdot \ldots \cdot (z_{n0}x_1)x_1x_n\\
&=& z_{n0}(x_1z_{n0}x_1^{-1})(x_1^{2}z_{n0}x_1^{-2})\cdot \ldots \cdot (x_1^{n-3}z_{n0}x_1^{3-n})x_1^{n-1}x_n\\
&=& z_{n0}z_{n1}\cdot \ldots \cdot z_{n,n-3}y_{n},\\
s^{-1}y_3s &=& x_n^{n-2}x_2x_n=(z_{n0}x_1)^{n-2}x_2x_n=(z_{n0}x_1)\cdot \ldots \cdot (z_{n0}x_1)x_2x_n\\
&=& z_{n0}z_{n1}\cdot \ldots \cdot z_{n,n-3}x_1^{n-2}x_2x_n=z_{n0}z_{n1}\cdot \ldots \cdot z_{n,n-3}z_{2,n-2}y_n.
\end{eqnarray*}

For all  $i=3, 4,\ldots,n$, on the same line,  we have
$$
s^{-1}y_i s=x_n^{n-2}x_{i-1}x_n=z_{n0}z_{n1}\cdot \ldots \cdot z_{n,n-3}z_{i-2,n-2}y_n.
$$


Finally we compute the power of $s$.
\begin{eqnarray*}
s^n &=& (\theta,x_1)^n=(\theta,x_1)(\theta,x_1)\cdot \ldots \cdot (\theta,x_1)\\
&=& (\theta^n,x_1\theta(x_1)\cdot \ldots \cdot \theta^{n-1}(x_1))=(1,x_1x_2 \cdot \ldots \cdot x_n)\\
&=& (1,(x_1x_2x_1^{-2})(x_1^2 x_3x_1^{-3}) \cdot \ldots \cdot (x_1^{n-2}x_{n-1}x_1^{1-n})(x_1^{n-1}x_{n}))\\
&=& z_{21}z_{32}\cdot \ldots \cdot z_{n-1,n-2}z_{n,n-1}.
\end{eqnarray*}
This completes the proof.
\end{proof}

We now get

\begin{cor}
The skew brace $(F_n, \cdot, \circ)$  constructed in Theorem \ref{cyclic1} is non-trivial and the group $( F_n, \circ )$ is free group of rank  $n$.
\end{cor}
\begin{proof}
That $(F_n, \cdot, \circ)$ is a non-trivial skew brace easily follows from the construction.  For the second assertion we  use the following theorem of Stallings \cite{St1968}:
If a finitely generated torsion free group  $G$ contains a free subgroup of finite index, then   $G$
is a free group.

The kernel of the homomorphism  $\lambda: F_n \rightarrow \mathrm{Aut}\,F_n$
is a free subgroup of  $F_n$ and has index  $n$.
Hence, it suffices to prove that
$( F_n, \circ )$ does not admit torsion elements.
Notice that the group  $( F_n, \circ )$ is equal to the subgroup
$$
H_{\lambda}=
\left\{ \, (\lambda(a),a) \, | \,  a \in   F_n   \, \right\}
$$
of $\mathrm{Hol}\, F_n$.

Let  $(\lambda(a),a)\in H_\lambda$ has  finite order.
Since $\theta^n=1$ and $\mathrm{Ker}\,\lambda$ does not have torsion,
we have $\lambda(a)=\theta^{m}$, $m \ne 0$ and $m$ divides  $n$.
Hence, $a=x_1^{m}w$ for some  $w\in \mathrm{Ker}\,\lambda$.
We have
$$
(\theta^{m}, x_1^{m}w)^{\frac{n}{m}}=(1, (x_1^{m}w) \theta^{m}(x_1^{m}w)\cdot \ldots
\cdot \theta^{m(\frac{n}{m}-1)}(x_1^{m}w)).
$$
Since $\theta (\mathrm{Ker}\,\lambda) \subseteq \mathrm{Ker}\,\lambda$, it follows that
$$
 x_1^{m} \theta^{m}(x_1^{m})\cdot \ldots \cdot \theta^{m(\frac{n}{m}-1)}(x_1^{m}) \in \mathrm{Ker}\,\lambda.
$$
Again, since $\theta (x_1) \equiv x_1 ( \mathrm{mod}\, \mathrm{Ker}\,\lambda)$,
$$
 x_1^{m(\frac{n}{m}-1)} \in \mathrm{Ker}\,\lambda,
$$
which means  $x_1^{n-m} \in \mathrm{Ker}\,\lambda$,  a contradiction, because
$\lambda (x_1^{n-m})=\theta^{n-m}\neq 1$. The proof is complete.
\end{proof}


\bigskip

We now construct a $\lambda$-cyclic skew brace on the free group $F_n$ of rank $n$ such that $\lambda(F_n)$ is infinite cyclic.

\begin{thm} \label{t4}
Let $F_n = \gen{x_1, \ldots, x_n}$ be the free group of rank $n$, 
 $w$ be a fixed non-identity element of  $F_n$ and $\theta: F_n \rightarrow F_n$
be the inner automorphism
$$
\theta: a \mapsto w a w^{-1},\quad a \in F_n.
$$
The homomorphism
$\lambda: F_n \rightarrow \mathrm{Aut}\,F_n$ is defined by the rule
$$
\lambda: x_i \mapsto \theta,\quad  i=1,\ldots,n.
$$

 Then in the  $\lambda$-cyclic skew brace $(F_n, \cdot, \circ)$ 
the multiplicative group $( F_n, \circ )$
is an extension of the free group  $F_{\infty}$ $(= \Ker \, \lambda)$ of infinite rank with free generators
$$
z_{j,k}=x_1^k x_j x_1^{-k-1}, \quad  k \in \mathbb{Z}, \quad j=2,\ldots,n,
$$
by the infinite cyclic group  $\langle s \rangle$, which acts by conjugations on the generators of
$F_{\infty}$ by the rule:
$$
s^{-1}z_{j,k}s= z_{j,k-m-1}, \quad k \in \mathbb{Z}, \quad j=2,\ldots,n,
$$
where $m$ is such that
$w\equiv x_1^m (\mathrm{mod}\, \mathrm{Ker}\,\lambda)$.

In particular, if $m\neq -1$, then we can express the generators  $z_{j,k}$ in terms of  generators
 $z_{j1},\ldots,z_{j,m+1}$, $j=2,\ldots,n$,
by the rule
$s^{-1}z_{j,k}s= z_{j,k-m-1}$, and so $\left\langle F_n, \circ \right\rangle \cong F_{r}$
is the free group of rank  $r=|(m+1)(n-1)|+1$.
If $m=-1$, then
$( F_n, \circ ) \cong F_{\infty} \times \mathbb{Z}$
is the direct product of the free group of infinite rank and the infinite cyclic group.
\end{thm}

\begin{proof}
Consider a subgroup
$$
H_{\lambda}=
\left\{ \, (\lambda(a),a) \, | \,  a \in   F_n   \, \right\}
$$
of $\mathrm{Hol}\, F_n$.
The kernel of the homomorphism  $\lambda$ consists of the words $w \in F_n$ such that
$l(w)=0.$ Notice that
$$
x_i^{-1}\lambda_{x_j} (x_i)=x_i^{-1}\theta(x_i)=x_i^{-1}wx_{i}w^{-1}.
$$
Hence, $x_i^{-1}\lambda_{x_j} (x_i)$
lies in the kernel of  $\lambda$ and $H_{\lambda}$ is a regular subgroup of $\mathrm{Hol}\, F_n$. Existence of the skew brace $(F_n, \cdot, \circ)$ now follows.

We find the generators of $\mathrm{Ker}\, \lambda$ using the Reidemeister-Shraier method as follows. We have

$\bullet$ $\left\{x_1,\ldots,x_n \right\}$ is a generating set of  $F_n$,

$\bullet$ $\left\{x_1^k,\,\, k \in \mathbb{Z} \right\}$ is a Shraier set of coset representatives of  $F_n/ \mathrm{Ker}\, \lambda$.

Then the non-trivial elements of the form
$x_1^k x_j (\overline{x_1^k x_j})^{-1}$,
$ k \in \mathbb{Z}$, $1=2,\ldots,n$ are free generators of  $\mathrm{Ker}\, \lambda$.
It means that  $\mathrm{Ker}\, \lambda$ has the basis
$$
z_{j,k}=x_1^k x_j x_1^{-k-1}, \quad k \in \mathbb{Z}, \quad j=2,\ldots,n.
$$

The quotient $H_{\lambda}/ \mathrm{Ker}\,\lambda \cong \mathbb{Z} $
and is generated, for example, by the images of  $s=(\theta, x_1)$.
We now compute the action of  $s$ on the generators of  $\mathrm{Ker}\,\lambda$.
Notice  that
$$
s^{-1}=(\theta, x_1)^{-1}=(\theta^{-1}, \theta^{-1}(x_1^{-1}))=(\theta^{-1}, w^{-1}x_1^{-1}w).
$$
Hence,
\begin{eqnarray*}
s^{-1}z_{j,k}s &=& (\theta^{-1}, w^{-1}x_1^{-1}w)(1,z_{j,k} ) (\theta, x_1)=
(\theta^{-1}, w^{-1}x_1^{-1}w)(\theta,z_{j,k} x_1)\\
&=& (1, w^{-1}x_1^{-1}w w^{-1} z_{j,k} x_1 w)=(1, w^{-1}x_1^{-1} z_{j,k} x_1 w)\\
&=& (1, w^{-1}z_{j,k-1} x_1 w).
\end{eqnarray*}
So,
$$
s^{-1}z_{j,k}s= w^{-1}z_{j,k-1} w, \quad k \in \mathbb{Z}, \quad j=2,\ldots,n.
$$
Note that if $w \in \mathrm{Ker}\,\lambda$, then it is a composition of the shift of the indexes on  $k$ and the inner automorphism of $\mathrm{Ker}\,\lambda$.

Now assume that
$$
w=x_1^m w_0,
$$
where  $m\in \mathbb{Z}$, $w_0\in \mathrm{Ker}\,\lambda$. Then
$$
s^{-1}z_{j,k}s= w_0^{-1}z_{j,k-m-1}  w_0, \quad k \in \mathbb{Z}, \quad j=2,\ldots,n.
$$
Hence, the element  $sw_0^{-1}$ acts on the generators  of $\mathrm{Ker}\,\lambda$
by the shift of the indexes  on $k$:
$$
(sw_0^{-1})^{-1}z_{j,k}(sw_0^{-1})= z_{j,k-m-1}, \quad k \in \mathbb{Z}, \quad j=2,\ldots,n.
$$
Thus
$$
( F_n, \circ )=H_\lambda=
\mathrm{Ker}\,\lambda \rtimes \mathbb{Z}\cong F_{\infty} \rtimes \mathbb{Z},
$$
where  $F_{\infty}=\langle  z_{j,k}, \quad k \in \mathbb{Z}, \quad j=2,\ldots,n \rangle$
is a free group of infinite rank, $\mathbb{Z}= \left\langle s \right\rangle$
is the infinite cyclic group  and the action is defined by the rules
$$
s^{-1}z_{j,k}s= z_{j,k-m-1}, \quad k \in \mathbb{Z}, \quad j=2,\ldots,n.
$$
The proof is complete.
\end{proof}

As a direct consequence of Theorem \ref{t4}, we get

\begin{cor}\label{s-last-cor}
There exists a symmetric linear  brace system $\mathcal{B}_I(F_n) = (F_n, \{\circ_i\}_{i \in I})$, $n \ge 2$, where $I$ is an infinite linearly ordered set.
\end{cor}

\begin{cor} \label{cor-rank}
For any natural number $n\geq 2$ there exists a homomorphism
$\lambda: F_2 \rightarrow \mathrm{Aut}\,F_2$, such that for the skew left brace $( F_2, \cdot , \circ )$ the multiplicative group
 $(F_2, \circ )$ is isomorphic to $F_n$.
\end{cor}

Let us define a subset of natural numbers consisting of
$$
r_0 = 2,~~ r_{k+1} = 2r_k + 1,~~ k = 0, 1, \ldots
$$
and define the graph $\Gamma$ with the set of vertices $V = \{v_k~|~ k = 0, 1, \ldots\}$, and set of edges
$$
E = \{ e_{k,k+1} = (k, k+1), e_{k+1,k} = (k+1, k)~|~ k = 0, 1, \ldots\}.
$$
 From Theorem \ref{t4} for $m=1$ it follows that we can construct brace system $\mathcal{B}_{\Gamma}(F_2)$ such that
$$
(F_2, \circ_k, \circ_{k+1}),~~k = 0, 1, \ldots
$$
is a symmetric skew brace such that $(F_2, \circ_k) \cong F_{r_k}$. Let $I = \{0\} \cup \mathbb{N}$. Hence, from this construction and Corollary~\ref{cor-rank}, we have

\begin{cor}\label{last-cor}
There is a symmetric linear brace system $\mathcal{B}_I(G) = (G, \{\circ_k\}_{ k \in I})$ such that  all groups $(G, \circ_k)$ are pairwise non-isomorphic.
\end{cor}

\bigskip

\section{Rota--Baxter operators and brace systems}

 L. Guo, H. Lang, Y. Sheng  \cite{Guo2020} defined the notion of Rota--Baxter operator
 on a group. A map $B\colon G\to G$ is called a \emph{Rota--Baxter operator} (of weight~1) if
$$
B(g)B(h) = B( g B(g) h B(g)^{-1} )
$$
 for all $g, h\in G$. A group $G$ with a Rota--Baxter operator $B$ is called a \emph{Rota--Baxter group}.

A new binary operation $\circ \colon G \to G$ on a Rota--Baxter group $(G,B)$
was defined in \cite{Guo2020} .

\begin{proposition}[\cite{Guo2020}]\label{prop:Derived}
Let $(G, \cdot, B)$ be a Rota--Baxter group.

a) The pair $(G, \circ )$ with the product
\begin{equation}\label{R-product}
g\circ h = gB(g)hB(g)^{-1},
\end{equation}
where $g,h\in G$, is also a group.

b) The operator $B$ is a Rota--Baxter operator on the group $(G,\circ)$.

c) The map $B\colon (G,\circ) \to (G,\cdot)$
is a homomorphism of Rota--Baxter groups.
\end{proposition}

As observed in \cite[Proposition 3.1]{BG-1}, given a Rota--Baxter group $(G, \cdot, B)$, we get a skew left brace $(G, \cdot, \circ)$, where  `$\circ$' is defined  \eqref{R-product}. Such a skew brace will be called a Rota--Baxter skew brace.
The following result provides a procedure for converting a word written in $(G, \circ)$ to a word written in $G$.

\begin{proposition}[\cite{BG}] \label{for}
Let $(G, B)$ be a Rota--Baxter group.
If $A$ is a subset of $G$ and
$$
w = a_{i_1}^{\circ(k_1)} \circ a_{i_2}^{\circ(k_2)} \circ \ldots
    \circ a_{i_s}^{\circ(k_s)},\quad a_{i_j} \in A,\ k_j \in \mathbb{Z},
$$
is presented by a word under the operation $\circ$, then
\begin{equation*}\label{WordInG_B}
w {=} (a_{i_1}B(a_{i_1}))^{k_1}
    (a_{i_2}B(a_{i_2}))^{k_2} \ldots
    (a_{i_s}B(a_{i_s}))^{k_s}
    B(a_{i_s})^{-k_s}
    B(a_{i_{s-1}})^{-k_{s-1}} \ldots
    B(a_{i_1})^{-k_1}.
\end{equation*}
In particular,
$a^{\circ(-1)} = B(a)^{-1} a^{-1} B(a)$.
\end{proposition}

\medskip

Suppose that  $(G, \cdot, \circ_B)$ is a Rota--Baxter skew brace. The next proposition gives condition under which this skew brace is a symmetric skew brace.

\begin{proposition}
A Rota--Baxter skew brace  $(G, \cdot, \circ)$ is symmetric if and only if $B$ is an anti-homomorphism  modulo   $Z(G^{(\cdot)})$, that is,
$$
B(c)^{-1} B(a)^{-1} B(ca) \in Z(G^{(\cdot)})
$$
for all $a, c \in G$. In particular, if $B$ is an anti-homomorphism, then $(G, \cdot, \circ)$ is a symmetric skew brace.
\end{proposition}

\begin{proof}
 As we know a skew brace is symmetric if and only if
$$
\lambda_{a\circ c}(b) = \lambda_{c \cdot a}(b),~~~a, b, c \in G.
$$
Rewrite the left side using part c) of Proposition \ref{prop:Derived}, we get 
$$
\lambda_{a\circ c}(b) = B(a \circ c) b B(a \circ c)^{-1} = B(a) B(c) b B(a)^{-1} B(c)^{-1}.
$$
Rewriting the right side gives
$$
\lambda_{c \cdot a}(b) = B(c a) b B(c a)^{-1}.
$$
Hence,
$$
[b, B(c)^{-1} B(a)^{-1} B(c a)] = 1.
$$
Since $b$ is arbitrary element of $G$, one gets  $B(c)^{-1} B(a)^{-1} B(c a) \in Z(G^{(\cdot)})$. 

Converse part is left as an exercise.
\end{proof}

\begin{cor}
If $B$ is an endomorphism of $G$ onto  some abelian subgroup of $G$, then  $(G, \cdot, \circ)$ is a symmetric  skew brace.
\end{cor}

The next lemma follows from the definition of the Rota--Baxter operator

\begin{lemma}
If a Rota--Baxter operator $B : G \to G$ is an anti-homomorphism, then
$$
[B(b), B^2(a) B(a)] = 1
$$
for all $a, b \in G$.
\end{lemma}

Connection between Rota--Baxter skew braces
and $\lambda$-homomorphic skew left braces is given by

\begin{proposition}[\cite{BG-1}]
A Rota--Baxter skew brace $(G,\cdot,\circ)$  is $\lambda$-homomorphic 
if and only if $B(ac)^{-1}B(a)B(c)\subseteq Z(G^{(\cdot)})$
for all $a,c\in G$.
\end{proposition}

\medskip

The following concept  of  a skew left multibrace was introduced in \cite{BG-1}.

\begin{definition}
Let $k$ be a~natural number.
By  a {\it skew left $k$-brace} we mean a $(k+1)$-groupoid
$(G, \circ_0, \circ_1, \ldots,\circ_k)$, i.\,e.,
a~non-empty set~$G$ with $k+1$ binary algebraic operations, such that

1) $(G, \circ_i)$ is a~group for all $i=0,1,\ldots,k$;

2) for $0 < i \leq k$
$$
a \circ_{i} (b \circ_{i-1} c)
 = (a \circ_{i} b) \circ_{i-1} a^{\circ_{i-1}(-1)} \circ_{i-1} (a \circ_{i} c),
$$
where $a^{\circ_{i-1}(-1)}$ is the inverse to $a$ in the group $(G, \circ_{i-1})$.
\end{definition}

A~skew left 1-brace is just the skew left brace.
By a {\it skew left multibrace} we mean a~skew left $k$-brace for some $k>1$. In our terminology skew left multibrace is a brace system.

On a Rota--Baxter group one can construct multibrace using the following proposition.

\begin{proposition}[\cite{BG-1}]
Let $(G, \cdot, B)$ be an RB-group and $k$ be a~natural number.
Define on the set~$G$ binary operations $\circ_1,\circ_2,\ldots,\circ_k$ as follows,
$$
x \circ_{i+1} y = x \circ_i B(x) \circ_i y \circ_i (B(x))^{\circ_i(-1)},
$$
where $\circ_0 = \cdot$. Then
$(G,\cdot,\circ_1, \circ_2,\ldots,\circ_k)$ is a~skew left $k$-brace.
\end{proposition}

If we take a skew brace $(G, \circ_i, \circ_{i+1})$, then
$$
\lambda^{(i)}_{x}(y) = B(x) \circ_i y \circ_i  B(x)^{\circ_i(-1)}.
$$
Just to demonstrate, let us rewrite first two new operations in  terms of original operation $\cdot$. We have
$$
x \circ y = x \circ_{1} y = x \cdot B(x) \cdot y \cdot (B(x))^{-1},
$$
and
$$
\lambda_{x}(y) = \lambda^{(0)}_{x}(y) = B(x) y B(x)^{-1}
$$

Further,
$$
x \circ_{2} y = x \circ_1 B(x) \circ_1 y \circ_1 (B(x))^{\circ_1(-1)},
$$
and by Proposition \ref{for}, using the equality $ (B(x))^{\circ_1(-1)} =  (B^2(x))^{-1} B(x)^{-1} B^2(x)$, we get
$$
x \circ_{2} y = x (B(x))^2 B^2(x) y B(y) (B^2(x))^{-1} B(x)^{-1} B^2(x) B(y)^{-1} (B^2(x))^{-1} B(x)^{-1}.
$$
In this case
\begin{eqnarray*}
\lambda^{(1)}_{x}(y) &=& B(x) \circ y \circ  B(x)^{\circ(-1)} = B(x) \circ y \circ  \left(  (B^2(x))^{-1} B(x)^{-1} B^2(x) \right) \\
&=& B(x) B^2(x) y B(y) (B^2(x))^{-1} B(x)^{-1} B^2(x) B(y)^{-1} (B^2(x))^{-1}.
\end{eqnarray*}

\medskip

Let $(G, \cdot)$ be a group. Then the map $B(g) = g^{-1}$, $g \in G$, is a Rota--Baxter operator on $G$. We see that $B$ is an anti-homomorphism. The skew brace constructed from the Rota--Baxter group $(G, \cdot, B)$ is the one as in Example \ref{example1}. We conclude with another concrete example.

\begin{example}
Let $G = F_2(x, y)$ be a free group of rank 2. The homomorphism $B : G \to \langle x \rangle$, which is defined on the generators by $B(x) = x$, $B(y) = x$ is a Rota--Baxter operator on $G$. Let us construct multi-brace which it defines. For $a, b \in G$, we define
$$
a \circ b =  a x^{l(a)} b x^{-l(a)},
$$
where $l(a)$ is as defined in the preceding section. Inductively, for any integer $m \ge 1$, we can define binary operations
$$
a \circ_m b =  a x^{ml(a)} b x^{-ml(a)}.
$$
It is not difficult to see that $(G, \cdot, \circ_1, \circ_2, \ldots)$ is a multibrace with infinitely many binary operations.
\end{example}

We close with the comment that connection between  Rota--Baxter  skew braces and skew braces constructed from $\gamma$-functions (closely related  to $\lambda$-homomorphic skew braces) in terms of second cohomology group has been very recently investigated in \cite{CS}.

\bigskip

\noindent{\it Acknowledgements.}  The first author is supported by the Ministry of Science and Higher Education of Russia (agreement No. 075-02-2021-1392). The first and second named authors acknowledge the support of RFBR (project No. 19-01-00569). The third named author acknowledges the support of DST-SERB Grant MTR/2021/000285. The authors are grateful to the unanimous referee for making several useful suggestions, which improved the quality of the manuscript. The authors thank Lorenzo Stefanello for pointed out a gap in an argument  (regarding step of poly-triviality of symmetric skew braces) at the end of Section 6 (of  earlier version) of this article (which we have modified now).


\begin{thebibliography}{999}

\bibitem{A}
B. Amberg and Y.P. Sysak,  \emph{On Associative Rings with Locally Nilpotent Adjoint Semigroup}. Commun.  Algebra, {\bf 31} (2003), 123-132.

\bibitem{BCJO18}
D. Bachiller,  F. Cedo,  E.  Jespers and J.  Okninski,  \emph{Iterated matched products of finite braces and simplicity; new solutions of the Yang--Baxter equation}, Trans. Amer. Math. Soc. {\bf 370} (2018), 4881-4907.

\bibitem{BNY}
 V. G. Bardakov, M. V. Neshchadim, M. K. Yadav,
\emph{Computing skew left braces of small orders}, Internat. J. Algebra Comput., {\bf 30} (2020), 839-851.

\bibitem{BNY22}
 V. G. Bardakov, M. V. Neshchadim, M. K. Yadav,
\emph{On $\lambda$-homomorphic skew  braces}, J. Pure  Appl. Algebra 226 (2022) 106961.

\bibitem{BG}
V.~G.~Bardakov, V. Gubarev,
\emph{Rota--Baxter operators on groups},
arXiv:2103.01848, 26~p.


\bibitem{BG-1}
V.~G.~Bardakov, V. Gubarev,
\emph{Rota--Baxter groups, skew left braces, and the Yang-Baxter equation}, J. Algebra {\bf 596}, 328-351 (2022).


\bibitem{Byott15}
 N. P. Byott, \emph{Solubility criteria for Hopf-Galois structures}, New York J. Math. {\bf 21} (2015), 883-903.

\bibitem{CC}
E. Campedel, A. Caranti, and I. Del Corso, \emph{Hopf--Galois structures on extensions of degree $p^2 q$ and skew braces of order $p^2 q$: the cyclic Sylow $p$-subgroup case},  J. Algebra {\bf 556} (2020), 1165-1210.

\bibitem{Ca}
A. Caranti, \emph{Bi-skew braces and regular subgroups of the holomorph}, J. algebra, {\bf 562} (2020), 647-665.

\bibitem{CS21a}
A. Caranti and L. Stefanello,  From endomorphisms to bi-skew braces, regular subgroups, the Yang-Baxter equation, and Hopf--Galois structures. J. Algebra {\bf 587} (2021), 462-487.

\bibitem{CS21b}
A. Caranti and L. Stefanello, Brace blocks from bilinear maps and liftings of endomorphisms,  https://arxiv.org/abs/2110.11028.



\bibitem{CS}
A. Caranti, L. Stefanello, \emph{Skew braces from Rota--Baxter operators: A cohomological characterisation, and some examples}, arXiv:2201.03936.


\bibitem{FC2018}
F. Cedo, \emph{Left braces: solutions of the Yang--Baxter equation}. Adv. Group Theory Appl. {\bf 5} (2018), 33-90.

\bibitem{CJR10}
F. Cedo, E. Jespers and A. del Rio, \emph{Involutive Yang--Baxter groups}. Trans. Amer. Math. Soc. {\bf 362} (2010),  2541-2558.

\bibitem{CSV19}
F. Cedo, A. Smoktunowicz and L. Vendramin, \emph{Skew left braces of nilpotent type}, Proc. London Math. Soc. {\bf 118} (2019), 1367-1392.


\bibitem{Chi}
L. N. Childs, \emph{Bi-skew braces and Hopf Galois structures}, New York
J. Math. {\bf 25} (2019), 574-588.


\bibitem{D1992}
V. Drinfeld, \emph{On some unsolved problems in quantum group theory}. Quantum groups (Leningrad, 1990),  Lecture Notes in Math.  {\bf 1510}, Springer, Berlin, 1992, pp. 1-8.


\bibitem{ESS99}
P. Etingof, T. Schedler and A. Soloviev,  \emph{Set-theoretical solutions to the quantum Yang-Baxter equation}. Duke Math. J.  {\bf 100} (1999),  169-209.

\bibitem{GV2017}
 L. Guarnieri  and  L. Vendramin, \emph{Skew braces and the Yang-Baxter equation},  Math. Comp. {\bf 86} (2017), 2519-2534.

\bibitem{Guo2020}
L. Guo, H. Lang, and Yu. Sheng,
\emph{Integration and geometrization of Rota--Baxter Lie algebras},
Adv. Math. {\bf 387} (2021), 107834.

\bibitem{JKAV21}
E. Jespers, L. Kubat,. A. Van Antwerpen,  L. Vendramin, \emph{Radical and weight of skew braces and their applications to structure groups of solutions of the Yang-Baxter equation}, Adv. Math. {\bf 385} (2021), Paper No. 107767, 20 pp.

\bibitem{Joyce}
D.~Joyce,
\emph{A~classifying invariant of knots, the knot quandle},
J.~Pure Appl. Algebra  {\bf 23} (1982), 37--65.

\bibitem{Koch20a}
A.~ Koch, \emph{Abelian maps, bi-skew braces, and opposite pairs of Hopf-Galois structures}, Proc. Amer. Math. Soc.  {\bf 8} (2021), 189-203.

\bibitem{Koch22}
A.~Koch, \emph{Abelian maps, brace blocks, and solutions to the Yang-Baxter equation}, J. Pure and Applied Algebra {\bf 226} (2022), 107047.

\bibitem{KT21}
A.~Koch and P.~J.~ Truman,  \emph{Opposite skew left braces and applications}, J. Algebra {\bf 546} (2020), 218-235.

\bibitem{KSV21}
A. Konovalov,  A. Smoktunowicz,  L. Vendramin,  \emph{On skew braces and their ideals}, Exp. Math. {\bf 30} (2021),  95-104.

\bibitem{K}
A. G. Kurosh,
\emph{General algebra. Lectures of 1969--1970 academic year}, M. Nauka, 1974 (in Russian).



\bibitem{Ma}
 A. I. Mal'cev, \emph{Generalized nilpotent algebras and their associated groups}, (Russian) Mat. Sbornik N.S. {\bf 25(67)}, (1949), 347-366.


\bibitem{Matveev}
S. Matveev,
\emph{Distributive groupoids in knot theory},
Mat. Sb. (N.S.), {\bf 119 (161)} (1982), 78-88 (in Russian).

 \bibitem{Nas}
 T. Nasybullov,  \emph{Connections between properties of the additive and the multiplicative groups of a two-sided skew brace},  J. Algebra {\bf 540} (2019), 156-167.



\bibitem{R2007}
W. Rump, \emph{Braces, radical rings and the  quantum Yang--Baxter equations}, J. Algebra {\bf 307} (2007), 153-170.

\bibitem{WR2019}
W. Rump, \emph{Classification of cyclic braces, II}, Trans. Amer. Math. Soc. {\bf 372} (2019), 305-328.

\bibitem{R19}
W.~Rump, \textit{A covering theory for non-involutive set-theoretic solutions to the Yang-Baxter equation}.  J. Algebra \textbf{520} (2019), 136-170.

\bibitem{AS2018}
A. Smoktunowicz,  \emph{On Engel groups, nilpotent groups, rings, braces and the Yang-Baxter equation}, Trans. Amer. Math. Soc. {\bf 370} (2018), 6535-6564.

\bibitem{SV}
A. Smoktunowicz and L. Vendramin,  \emph{On skew braces (with an appendix by N. Byott and L. Vendramin)}, J. Comb. Algebra {\bf 2} (2018), 47-86.


\bibitem{ST22}
L. Stefanello and S. Trappeniers, On bi-skew braces and brace blocks,  https://arxiv.org/pdf/2205.15073.pdf.


 \bibitem{St1968}
J. R. Stallings, \emph{On torsion-free groups with infinitely many ends},
Ann. Math. {\bf 88} (1968), 312-334.

 \bibitem{Sw1969}
R. G. Swan, \emph{Groups of cohomological dimension one},
J. Algebra {\bf 12} (1969), 585-601.

\end{thebibliography}
\end{document}